\documentclass{article}
\usepackage{amsmath,amssymb,bm}

\makeatletter
    \def\@begintheorem#1#2{\trivlist
       \item[\hskip \labelsep{\bfseries #1\ #2.}]\itshape}
    \def\@opargbegintheorem#1#2#3{\trivlist
       \item[\hskip \labelsep{\bfseries #1\ #2.\i#3)}]
       \itshape}
  \makeatother
  
\newcommand{\qed}{\hfill \ensuremath{\Box}}
\newcommand{\Aut}{\mathrm{Aut}_{\mathrm{cfo}}^n(\mathfrak{g})}
\newcommand{\Auta}{\mathrm{Aut}_{\mathrm{cfo}}^n(\mathcal{A})}

\newcommand{\Lmh}{L_{\bm{m}}(\mathfrak{g}, \bm{\sigma}, \mathfrak{h})}
\newcommand{\Lmhb}{L_{\bm{m}'}(\mathfrak{g}', \bm{\sigma}', {\mathfrak{h}'})}

\newtheorem{dfn}{Definition}[subsection]
\newtheorem{thm}[dfn]{Theorem}
\newtheorem{lem}[dfn]{Lemma}
\newtheorem{rem}[dfn]{Remark}
\newtheorem{prop}[dfn]{Proposition}
\newtheorem{cor}[dfn]{Corollary}
\newtheorem{con}[dfn]{Construction}
\newcommand{\gn}{\mathfrak{g}}

\newcommand{\hz}{\mathfrak{h}}
\newcommand{\Lamm}{\bar{\Lambda}_{\bm{m}}}

\newcommand{\lam}{\bar{\lambda}}
\newcommand{\bmu}{\bar{\mu}}
\newcommand{\xal}{x_{\alpha}^{\bar{\lambda}}}
\newcommand{\mxal}{x_{-\alpha}^{-\bar{\lambda}}}
\newcommand{\sltwo}{\mathfrak{sl}_2(k)}
\newcommand{\auttheta}{\theta_{\alpha}^{\bar{\lambda}}}

\newcommand{\hzb}{\mathfrak{h}'}

\newcommand{\sigmas}{(\sigma_1,\dots,\sigma_n)}

\newcommand{\ms}{(m_1,\dots,m_n)}
\newcommand{\idcond}[2]{\bm{#1}^{\bm{#2}}=\mathbf{id}}
\newcommand{\gsigma}{\gn^{\bm{\sigma}}}

\newcommand{\lr}[2]{\bigl\langle \mathrm{supp}_{#1}(#2) \bigr\rangle} 
\renewcommand{\L}{\mathfrak{L}}

\newcommand{\B}{\mathfrak{B}}
\newcommand{\s}{\mathfrak{s}}
\newcommand{\Lb}{\mathcal{L}}
\newcommand{\C}{\mathfrak{C}}
\newcommand{\D}{\mathfrak{D}}
\newcommand{\Ab}{\mathcal{A}}
\newcommand{\Bb}{\mathcal{B}}
\newcommand{\igisom}{\cong_{\mathrm{ig}}}
\newcommand{\suppisom}{\cong_{\mathrm{supp}}}
\newcommand{\Mm}{M_{\bm{m}}(\Ab, \bm{\sigma})}
\newcommand{\Mmb}{M_{\bm{m'}}(\Ab', \bm{\sigma}')}
\newcommand{\zisom}{\cong_{\mathbb{Z}^n-\mathrm{su}}}

\begin{document}

\title{Multiloop Lie algebras and the construction of extended affine Lie algebras}
\author{Katsuyuki Naoi}
\date{}

\maketitle

\begin{abstract}
  It is known that a multiloop Lie algebra, which is constructed using multiloop realization,
  can be a Lie $\mathbb{Z}^n$-torus 
  if a given multiloop Lie algebra satisfies several conditions, 
  and it is also known that a family of extended affine Lie algebras (EALAs) is obtained 
  from a Lie $\mathbb{Z}^n$-torus.
  In many cases, however, even if a given multiloop Lie algebra does not satisfy these conditions,
  we can also construct a family of EALAs from it.
  In this paper, we study this construction, and prove that two families of EALAs constructed from
  two multiloop Lie algebras coincide up to isomorphisms as EALAs if and only if 
  two multiloop Lie algebras are ``support-isomorphic''.
  Also, we give a necessary and sufficient condition for two multiloop Lie algebras to be support-isomorphic.
\end{abstract}
\section{Introduction}
The multiloop realization is introduced in \cite{MR2418198}:
from an algebra $\Ab$ that is not necessarily associative or unital, a finite sequence of mutually commutative finite order
automorphisms $\bm{\sigma}=\sigmas$, and a sequence of positive integers $\bm{m}=\ms$ 
such that $\sigma_i^{m_i}=\mathrm{id}$ for $1 \le i \le n$, we can construct
a $\mathbb{Z}^n$-graded algebra $M_{\bm{m}}(\Ab, \bm{\sigma})$ called a multiloop algebra.

We consider the case where an algebra $\Ab$ is a finite dimensional simple Lie algebra $\gn$, and we assume that 
$\gsigma := \{ g \in \gn \mid \sigma_i(g)=g \ \mathrm{for \ all} \ i \} \neq \{ 0 \}$.
In this case since $\gsigma$ is reductive,
we can consider a root space decomposition of $M_{\bm{m}}(\gn, \bm{\sigma})$ 
with respect to a Cartan subalgebra $\hz$ in $\gsigma$, and then we can see
$M_{\bm{m}}(\gn, \bm{\sigma})$ as a $Q_{\hz} \times \mathbb{Z}^n$-graded Lie algebra where
$Q_{\hz}$ is a root lattice.
In this paper, we call the $Q_{\hz} \times \mathbb{Z}^n$-graded Lie algebra a multiloop Lie algebra,
and denote it by $\Lmh$.
In \cite{MR2506428}, the authors have proved that $\Lmh$ can be a Lie $\mathbb{Z}^n$-torus if $\bm{\sigma}$
satisfies some conditions (the principal condition is that $\gsigma$ is a simple Lie algebra),
and in that case it is called a multiloop Lie $\mathbb{Z}^n$-torus.
A Lie $\mathbb{Z}^n$-torus is a $Q \times \mathbb{Z}^n$-graded Lie algebra,
where $Q$ is a root lattice of an irreducible finite root system, satisfying several axioms.
E.\ Neher has proved in \cite{MR2083842} that if a centreless Lie $\mathbb{Z}^n$-torus is given,
we can construct a family of extended affine Lie algebras (EALAs, for short).
However, unless $\gsigma = \{ 0 \}$, we can construct a family of EALAs from $\Lmh$
even if $\bm{\sigma}$ does not satisfy the condition for $\Lmh$ to be a Lie $\mathbb{Z}^n$-torus.
This fact can be seen by proving that the $Q_{\hz}$-support of $\Lmh$ with respect to $\hz$ is an
irreducible finite root system.
In this paper, we study this construction of a family of EALAs from a multiloop Lie algebra.

In \cite{MR2743759}, it has been proved that there exists a one-to-one correspondence 
between centreless Lie $\mathbb{Z}^n$-tori up to isotopy and 
families of EALAs up to isomorphism, where isotopy is an equivalence relation 
on a class of Lie $\mathbb{Z}^n$-tori defined in \cite{MR2506428}.
In this paper, we see that the similar result is obtained in the case of multiloop Lie algebras;
we define an equivalence relation ``support-isomorphic'' on the class of multiloop Lie algebras
(see Definition \ref{supp-isom}), and then we prove that two families of EALAs constructed from two multiloop Lie algebras
coincide up to isomorphism if and only if two multiloop Lie algebras are support-isomorphic.
Also, we give a necessary and sufficient condition for two multiloop Lie algebras to be support-isomorphic.

As we prove in Theorem \ref{Main theorem}, a multiloop Lie algebra $\Lmh$ is support-isomorphic 
to some Lie $\mathbb{Z}^n$-torus if and only if $\gsigma \neq \{ 0 \}$.
From this fact, we can see that the class of EALAs which can be constructed from multiloop Lie algebras 
coincides with that constructed from multiloop Lie $\mathbb{Z}^n$-tori.
It is, however, expected that, at least in some cases, considering whole multiloop Lie algebras makes
it easy to study the classification problem of EALAs.

We briefly outline the contents of this paper.
In section 2, we recall the definition and some results of multiloop algebras,
and define support-isomorphism.
In section 3, we define a multiloop Lie algebra $\Lmh$, and study the properties of the support of 
a $Q_{\hz}$-grading.
In section 4, we study a support-isomorphism of multiloop Lie algebras.
In section 5, we give a necessary and sufficient condition for a multiloop Lie algebra
to be support-isomorphic to some Lie $\mathbb{Z}^n$-torus, and finally, we study
the construction of EALAs from a multiloop Lie algebra. \\ \\  
\textbf{Assumptions and Notation.} 
\begin{enumerate}
  \item[(a)] Throughout this paper all vector spaces and algebras are defined over a base field $k$ 
             of characteristic 0 and we assume that $k$ is \textit{algebraically closed}.
             In this paper an algebra is not necessarily associative or unital. 
  \item[(b)] For each $n \in \mathbb{Z}_{>0}$, we choose a primitive $n$-th root of unity $\zeta_n\in k$  
             satisfying the following condition: for all $m,n \in \mathbb{Z}_{>0}$,
             \begin{equation}\label{46}
               \zeta_{mn}^{m}=\zeta_{n}. 
             \end{equation} 
  \item[(c)] For an n-tuple of positive integers $\bm{m}=(m_1,\dots,m_n)$, let 
             \[ \Lamm = \mathbb{Z}/m_1\mathbb{Z}\times\dots\times\mathbb{Z}/m_n\mathbb{Z}. 
             \] 
  \item[(d)] For a group $\Lambda$ and a subset $S \subseteq \Lambda$, 
             let $\langle S \rangle$ be a subgroup of $\Lambda$ generated by $S$. 
  \item[(e)] If $\Bb=\bigoplus_{\lambda \in \Lambda} \Bb^{\lambda}$
             is a $\Lambda$-graded algebra for some abelian group $\Lambda$, we put
             \[ \mathrm{supp}_{\Lambda}(\Bb) = \{ \lambda \in \Lambda \mid 
                \Bb^{\lambda}\neq \{ 0 \} \} \subseteq \Lambda. 
             \]
\end{enumerate}
\section{Multiloop algebras} \label{section1}
Although we are interested only in Lie algebras, we deal with general algebras in this section.
\subsection{Definitions and some results}
First, we recall the following basic definitions.
%
\begin{dfn} \label{centroid} \normalfont
  Suppose that $\Ab$ is an algebra.  
  \begin{enumerate}
    \item[(a)]
      Let $C(\Ab)$ be the subalgebra of $\mathrm{End}_{k}(\Ab)$ consisting of the $k$-linear endomorphisms
      of $\Ab$ that commute with all left and right multiplications by elements of $\Ab$.
      We call $C(\Ab)$ the \textit{centroid} of $\Ab$. 
    \item[(b)] 
      We say $\Ab$ is \textit{central-simple} if $\Ab$ is simple and $C(\Ab)=k \cdot \mathrm{id}$.
  \end{enumerate}
\end{dfn}
%
%
\begin{dfn} \label{def of graded-simple} \normalfont
  Let $\Lambda$ be an abelian group and $\Bb= \bigoplus_{\lambda \in \Lambda}\Bb^{\lambda}$ 
  be a $\Lambda$-graded algebra. 
  \begin{enumerate}
    \item[(a)] We say $\Bb$ is \textit{graded-simple} 
               if $\Bb \Bb \neq \{ 0 \}$ and graded ideals of $\Bb$ are only $\{ 0 \}$ and $\Bb$. 
    \item[(b)] Suppose that $\Bb$ is graded-simple.
               Then $C(\Bb)=\bigoplus_{\lambda \in \Lambda}C(\Bb)^{\lambda}$ 
               is a unital commutative associative $\Lambda$-graded algebra where
               \[ C(\Bb)^{\lambda}=\{ c \in C(\Bb) \mid c\Bb^{\mu} 
               \subseteq \Bb^{\lambda+\mu} \quad \mathrm{for} \ \mu \in \Lambda \},
               \]
               and $\Gamma_{\Lambda}(\Bb):=\mathrm{supp}_{\Lambda}(C(\Bb))$
               is a subgroup of $\Lambda$ \cite[Proposition 2.16]{MR2193194}.
               We call $\Gamma_{\Lambda}(\Bb)$ the \textit{central grading group} of $\Bb$.
               We say $\Bb$ is \textit{graded-central-simple} if $\Bb$ is graded-simple 
               and $C(\Bb)^0=k \cdot \mathrm{id}$.
  \end{enumerate}
\end{dfn}  
%
%
\begin{dfn} \label{dfn of two isom} \normalfont
  Let $\Lambda, \Lambda'$ be abelian groups. 
  \begin{enumerate}
    \item[(a)]
      Suppose that $\Bb$ and $\Bb'$ are $\Lambda$-graded algebras. 
      Then we say $\Bb$ and $\Bb'$ are \textit{$\Lambda$-graded-isomorphic} 
      if there exists an algebra isomorphism $\varphi: \Bb \to \Bb'$ such that 
      \[ \varphi(\Bb^{\lambda})={\Bb'}^{\lambda}
      \]
      for $\lambda \in \Lambda$.
      In that case, we call $\varphi$ a \textit{$\Lambda$-graded-isomorphism},
      and we write $\Bb \cong_{\Lambda} \Bb'$. 
    \item[(b)]
      Suppose that $\Bb$ is a $\Lambda$-graded algebra and $\Bb'$ is a $\Lambda'$-graded algebra. 
      Then we say $\Bb$ and $\Bb'$ are \textit{isograded-isomorphic} 
      if there exist an algebra isomorphism $\varphi: \Bb \to \Bb'$ 
      and a group isomorphism $\varphi_{\Lambda}: \Lambda \to \Lambda'$ such that 
      \[ \varphi(\Bb^{\lambda})={\Bb'}^{\varphi_{\Lambda}(\lambda)}
      \]
      for $\lambda \in \Lambda$. In that case we call $\varphi$ an \textit{isograded-isomorphism},
      and we write $\Bb \cong _{\mathrm{ig}} \Bb'$. 
  \end{enumerate}
\end{dfn}
%

To define a multiloop algebra, we use the following notation.
Suppose that $\Ab$ is an algebra.
We denote the set of $n$-tuples of commuting finite order automorphisms of $\Ab$
\[ \{(\sigma_1,\dots,\sigma_n)\in\mathrm{Aut}(\Ab)^n \mid \sigma_i\sigma_j=\sigma_j\sigma_i, \
   \mathrm{ord}(\sigma_i)<\infty \ \ \mathrm{for \ all} \ i,j \}
   \]
by $\mathrm{Aut}_{\mathrm{cfo}}^{n}(\Ab)$.
For $\bm{\sigma}=\sigmas \in \Auta$, we put 
\[ \Ab^{\bm{\sigma}}= \{ u \in \Ab \mid \sigma_i(u) = u \quad \mathrm{for} \ 1 \le i \le n \},
\]
and we write $\mathrm{ord}(\bm{\sigma})=(\mathrm{ord}(\sigma_1),\dots,\mathrm{ord}(\sigma_n)) \in
\mathbb{Z}_{>0}^n$.

A multiloop algebra has been defined in \cite{MR2418198} as follows:
%
\begin{dfn} \label{multiloop algebra} \normalfont
  Suppose that $\Ab$ is an algebra. 
  Let $n\in \mathbb{Z}_{>0}$, and assume that
  $\bm{\sigma}=(\sigma_1,\dots,\sigma_n)\in\Auta$ 
  and $\bm{m}=(m_1,\dots,m_n) \in \mathbb{Z}^n_{>0}$ satisfy
  \[ \sigma_i^{m_i}=\mathrm{id} \quad \mathrm{for} \quad 1 \le i \le n.
  \]  
  (Henceforth, we write $\idcond{\sigma}{m}$ to denote this condition). 
  Note that we do not necessarily assume that each $m_i$ is an order of $\sigma_i$.
  For $\lambda =(l_1, \dots ,l_n) \in \mathbb{Z}^n$, let 
  \[ \bar{\lambda}=(\bar{l_1},\dots ,\bar{l_n}) \in
  \Lamm(=\mathbb{Z}/m_1\mathbb{Z}\times\dots\times\mathbb{Z}/m_n\mathbb{Z})
  \]
  be the image of $\lambda$ under the canonical group homomorphism from $\mathbb{Z}^n$ onto $\Lamm$.
  Using $\bm{\sigma}$ and $\bm{m}$, we define a $\bar{\Lambda}_{\bm{m}}$-grading on 
  $\Ab$ as follows:
  for $\bar{\lambda}=(\bar{l_1},\dots,\bar{l_n})\in\bar{\Lambda}_{\bm{m}}$, 
  \begin{equation}\label{1}
    \Ab^{\bar{\lambda}_{(\bm{\sigma},\bm{m})}}=\{u\in\Ab \mid \sigma_{i}(u)
  =\zeta_{m_i}^{l_i}u \quad \mathrm{for} \ 1\le i \le n \}.
  \end{equation}
  (We usually use a notation $\Ab^{\bar{\lambda}}$ instead of $\Ab^{\bar{\lambda}_{(\bm{\sigma},\bm{m})}}$
  when it is obvious from the context that $\Ab$ is graded using $\bm{\sigma}$ and $\bm{m}$). 
  Then we can define a $\mathbb{Z}^n$-graded algebra 
  \begin{equation} \label{57}
    M_{\bm{m}}(\Ab, \bm{\sigma})=\bigoplus_{\lambda \in \mathbb{Z}^n}
    \Ab^{\bar{\lambda}} \otimes t^{\lambda}   \subseteq \Ab \otimes k[t_1^{\pm1},\dots,t_n^{\pm1}]
  \end{equation}
  where for $\lambda=(l_1,\dots,l_n)$, we put $t^{\lambda}=t_1^{l_1}t_2^{l_2}\dots t_n^{l_n}$.
  We call the $\mathbb{Z}^n$-graded algebra $M_{\bm{m}}(\Ab, \bm{\sigma})$ the \textit{multiloop algebra}
  of $\bm{\sigma}$ (based on $\Ab$ and relative to $\bm{m}$).
  We call $n$ the \textit{nullity} of $M_{\bm{m}}(\Ab, \bm{\sigma})$.
\end{dfn}
%

By \cite[Proposition 8.2.2]{MR2418198}, we have the following:
%
\begin{lem} \label{central grading group}
  Suppose that $\Ab$ is a central-simple algebra, 
  and $\Bb=M_{\bm{m}}(\Ab, \bm{\sigma})$ is a multiloop algebra of $\bm{\sigma} \in \Auta$ 
  relative to $\bm{m}=\ms \in \mathbb{Z}^n_{>0}$ where $\bm{\sigma}^{\bm{m}}= \bm{\mathrm{id}}$.
  Then $\Bb$ is a graded-central-simple $\mathbb{Z}^n$-graded algebra, and
  \[ \Gamma_{\mathbb{Z}^n}(\Bb)=m_1\mathbb{Z} \times \dots \times m_n \mathbb{Z} \subseteq \mathbb{Z}^n. 
  \]
  where $\Gamma_{\mathbb{Z}^n}(\Bb)$ is the central grading group of $\Bb$. 
  In particular, the rank of $\Gamma_{\mathbb{Z}^n}(\Bb)$ is $n$. 
\end{lem}
%

We use the following notation.
Let $\Ab$ be an algebra and $\bm{\sigma}=\sigmas \in \Auta$.  
For $P = (p_{ij})\in \mathrm{GL}_n(\mathbb{Z})$, we set 
\[ \bm{\sigma}^{P}=(\prod_{1 \le i \le n} \sigma_i^{p_{i1}},
   \prod_{1 \le i \le n} \sigma_i^{p_{i2}}, \dots ,\prod_{1 \le i \le n} \sigma_i^{p_{in}}).
\]
Since $\sigma_i$'s commute with each other and each $\sigma_i$ has a finite order, 
$\bm{\sigma}^P \in \Auta$. 
It is easy to check that 
$(\bm{\sigma}^P)^Q=\bm{\sigma}^{PQ}$ for $P,Q \in \mathrm{GL}_n(\mathbb{Z})$.
Therefore, $P: \bm{\sigma} \mapsto \bm{\sigma}^P$ 
defines a right $\mathrm{GL}_n(\mathbb{Z})$-action on $\Auta$. 
If $\Ab'$ is another algebra and $\varphi: \Ab \to \Ab'$ is an algebra isomorphism, we write
\[ \varphi\bm{\sigma} \varphi^{-1}=(\varphi\sigma_1\varphi^{-1},\dots,\varphi\sigma_n\varphi^{-1})
   \in \mathrm{Aut}_{\mathrm{cfo}}^n(\Ab').
\]

The following definition is introduced in \cite[Definition 8.1.1]{MR2418198} (In the definition, we let 
$\mathrm{diag}(a_1, \dots, a_n)$ denote an $n$-diagonal matrix
with the diagonal entries $(a_1, \dots, a_n)$):
%
\begin{dfn}\label{def4}  \normalfont
For $\bm{m}=(m_1,\dots,m_n) \in \mathbb{Z}_{>0}^n$ and $\bm{m}'=(m_1',\dots,m_n') \in \mathbb{Z}_{>0}^n$,
we set $D_{\bm{m}}=\mathrm{diag}(m_1,\dots,m_n)$,
$D_{\bm{m}'}=\mathrm{diag}(m_1',\dots,m_n')$.
For $P \in \mathrm{GL}_{n}(\mathbb{Z})$, 
we say that $P$ is \textit{$(\bm{m}', \bm{m})$-admissible} 
if $D_{\bm{m}'}{}^t \! PD_{\bm{m}}^{-1} \in \mathrm{GL}_n(\mathbb{Z})$ 
where ${}^t \! P$ is a transpose of $P$.
\end{dfn}
%
%
\begin{prop} \label{realization theorem}
  Suppose that $\Ab$ and $\Ab'$ are central-simple algebras.  
  Assume that $\bm{\sigma} \in \Auta$, $\bm{\sigma}' \in 
  \mathrm{Aut}_{\mathrm{cfo}}^n(\mathcal{A}')$ and $\bm{m}, \bm{m}' \in \mathbb{Z}_{>0}^n$ satisfy
  $\idcond{\sigma}{m}, {\bm{\sigma}'}^{\bm{m}'}= \bm{\mathrm{id}}$.
  Let $\Bb=M_{\bm{m}}(\Ab, \bm{\sigma}), \Bb'=M_{\bm{m}'}(\Ab', \bm{\sigma}')$.
  Then the following two statements are equivalent: 
  \begin{enumerate}
    \item[\textup{(a)}] $\Bb \cong_{\mathrm{ig}} \Bb'$. 
    \item[\textup{(b)}] 
      There exist a matrix $P \in \mathrm{GL}_n(\mathbb{Z})$ and an algebra isomorphism
      $\varphi: \Ab \to \Ab'$ such that $P$ is $(\bm{m}', \bm{m})$-admissible and 
      \begin{equation} \label{0}
        \bm{\sigma}'= \varphi \bm{\sigma}^P \varphi^{-1}. 
      \end{equation}
  \end{enumerate}
  Moreover, if $P$ and $\varphi$ satisfy \textup{(b)},
  we can take an isograded-isomorphism $\psi: \Bb \to \Bb'$
  satisfying $\psi(x \otimes 1)=\varphi(x) \otimes 1$ for $x \in \Ab^{\bm{\sigma}}$. 
\end{prop}
\textbf{Proof.}
  The first statement is \cite[Theorem 8.3.2 (ii)]{MR2418198}.
  Suppose that $P, \varphi$ satisfy (b), 
  and let $Q=D_{\bm{m}'}{}^t\!PD_{\bm{m}}^{-1} \in \mathrm{GL}_n(\mathbb{Z})$.
  If we define $\psi: \Bb \to \Bb'$ as 
  \[ \Bb^{\lambda} \ni x \otimes t^{\lambda} \mapsto \varphi(x) \otimes t^{\lambda {}^t \! Q} 
     \in \Bb'^{\lambda {}^t \! Q}
  \]
  for $\lambda \in \mathbb{Z}^n, x \in \Ab^{\bar{\lambda}}$, 
  then $\psi$ is an isograded-isomorphism by \cite[Proposition 8.2.1]{MR2418198}. 
  Clearly $\psi(x \otimes 1)=\varphi(x) \otimes 1$ for $x \in \Ab^{\bm{\sigma}}$.
\qed
%
%
%
\subsection{Support-isomorphism}
Let $\Bb$ be a $\Lambda$-graded algebra for an abelian group $\Lambda$, and
take a subgroup $\Lambda_{\mathrm{sub}} \subseteq \Lambda$ such that 
$\lr{\Lambda}{\Bb} \subseteq \Lambda_{\mathrm{sub}}$.
Since $\Bb=\oplus_{\lambda \in \Lambda_{\mathrm{sub}}} \Bb^{\lambda}$,
we can consider $\Bb$ canonically as a $\Lambda_{\mathrm{sub}}$-graded algebra.
In particular, we can view $\Bb$ as $\lr{\Lambda}{\Bb}$-graded.
%
\begin{dfn} \label{supp-isom} \normalfont
  Let $\Lambda, \Lambda'$ be abelian groups, and 
  suppose that $\Bb$ is a $\Lambda$-graded algebra and $\Bb'$ is a $\Lambda'$-graded algebra.
  We say $\Bb$ and $\Bb'$ are \textit{support-isograded-isomorphic} (or \textit{support-isomorphic}, 
  for short) if there exist an algebra isomorphism $\varphi: \Bb \to \Bb'$ and a group isomorphism 
  $\varphi_{\mathrm{su}}: \lr{\Lambda}{\Bb} \to \lr{\Lambda'}{\Bb'}$ such that 
  \[ \varphi(\Bb^{\lambda}) = {\Bb'}^{\varphi_{\mathrm{su}}(\lambda)}
  \]
  for $\lambda \in \lr{\Lambda}{\Bb}$:
  in other words, if $\Bb$ is considered as $\lr{\Lambda}{\Bb}$-graded and $\Bb'$ as 
  $\lr{\Lambda'}{\Bb'}$-graded, then $\Bb$ and $\Bb'$ are isograded-isomorphic.
  In that case, we call $\varphi$ a \textit{support-isograded-isomorphism} 
  (or \textit{support-isomorphism}, for short), and we write $\Bb \cong_{\mathrm{supp}} \Bb'$.
\end{dfn} 
%
The following lemma is obvious from the definitions:
%
\begin{lem} \label{supp and ig}
  Let $\Lambda, \Lambda'$ be abelian groups, 
  and suppose that $\Bb$ is a $\Lambda$-graded algebra and $\Bb'$ is a $\Lambda'$-graded algebra. 
  \begin{enumerate}
    \item[\textup{(a)}] If $\Bb \igisom \Bb'$, then $\Bb \suppisom \Bb'$. 
    \item[\textup{(b)}]
      If $\lr{\Lambda}{\Bb}=\Lambda$ and $\lr{\Lambda'}{\Bb'}= \Lambda'$, 
      then $\Bb \igisom \Bb'$ is equivalent to $\Bb \cong_{\mathrm{supp}} \Bb'$. 
  \end{enumerate} 
\end{lem}
%

We would like to give a necessary and sufficient condition for two multiloop algebras based on 
central-simple algebras to be support-isomorphic. To do this, we need the following lemmas.
%
\begin{lem} \label{Z-span} 
  Let $\Ab$ be an algebra, $\bm{\sigma} \in \Auta$.
  Then there exists $P \in \mathrm{GL}_n(\mathbb{Z})$ such that 
  \begin{equation} \label{59}
    \lr{\mathbb{Z}^n}{M_{\mathrm{ord}(\bm{\sigma}^P)}(\Ab, \bm{\sigma}^P)}=\mathbb{Z}^n.
  \end{equation}
\end{lem} 
{\textbf{Proof.}}
  Let $G= \langle \{ \sigma_1, \dots, \sigma_n \} \rangle$.
  By \cite[Proposition 5.1.3]{MR2506428}, there exists $P \in \mathrm{GL}_n(\mathbb{Z})$ such that
  \begin{equation} \label{60}
    |G|=\prod_{i=1}^n \mathrm{ord} \left( (\bm{\sigma}^P)_i \right),
  \end{equation}
  where $|G|$ denotes the cardinal number of $G$ and $\bm{\sigma}^P=((\bm{\sigma}^P)_1,\dots, (\bm{\sigma}^P)_n)$.
  By \cite[Lemma 3.2.4]{MR2418198}, (\ref{60}) is equivalent to (\ref{59}).
\qed
%

%
%
\begin{lem} \label{lem10}
  Let $\Ab$ be an algebra and $\Bb=\Mm$ be a multiloop algebra of nullity $n$.
  Then it follows that 
  \[ \Bb \suppisom M_{\mathrm{ord}(\bm{\sigma})}(\Ab, \bm{\sigma}).
  \]
\end{lem}
{\textbf{Proof.}} 
  Let $a_i \in \mathbb{Z}_{>0}$ be a positive integer 
  such that $\mathrm{ord}(\sigma_i)=m_i/a_i $ for $1 \le i \le n$.
  We write $\Bb_{\mathrm{ord}(\bm{\sigma})} = M_{\mathrm{ord}(\bm{\sigma})}(\Ab, \bm{\sigma})$.
  Let $f_{\bm{a}}:\mathbb{Z}^n \to \mathbb{Z}^n$ be an injective homomorphism defined as
  \[ f_{\bm{a}} \big( (l_1,\dots ,l_n) \big) =(a_1 l_1,\dots ,a_n l_n)
  \]
  for $(l_1,\dots ,l_n) \in \mathbb{Z}^n$.
  By (\ref{46}), we have $\zeta_{m_i}^{a_i}=\zeta_{\mathrm{ord}(\sigma_i)}$.
  Using this we have
  \begin{equation} \label{46'}
    \begin{split}
       \Ab^{\bar{\lambda}_{(\bm{\sigma}, \mathrm{ord}(\bm{\sigma}))}}&= \{ u \in \Ab \mid 
       \sigma_i(u)=\zeta_{\mathrm{ord}(\sigma_i)}^{{l_i}}u \ \ \mathrm{for} \ 1 \le i \le n \} \\
       &=\{u \in \Ab \mid\sigma_i (u) = \zeta_{m_i}^{a_i l_i}u \quad \mathrm{for} \ 1 \le i \le n \} \\
       &=\Ab^{\overline{f_{\bm{a}}(\lambda)}_{(\bm{\sigma},\bm{m})}}
     \end{split} 
  \end{equation}
  for $\lambda=(l_1, \dots, l_n) \in \mathbb{Z}^n$.
  Next, suppose that $\lambda=(l_1, \dots, l_n) \notin \mathrm{Im} \, f_{\bm{a}}$.
  Then there exists $j$ such that $a_j\nmid l_j$, 
  and from this we have $\Ab^{\bar{\lambda}_{(\bm{\sigma}, \bm{m})}} = \{ 0 \}$.
  Consequently, we can define an algebra isomorphism $\varphi: \Bb_{\mathrm{ord}(\bm{\sigma})} \to \Bb$ as 
  \[  \Bb_{\mathrm{ord}(\bm{\sigma})}^{\lambda}=
      \Ab^{{\bar{\lambda}}_{(\bm{\sigma}, \mathrm{ord}(\bm{\sigma}))}} \otimes t^{\lambda}
      \ni u \otimes t^{\lambda} \mapsto u \otimes t^{f_{\bm{a}} (\lambda)}
      \in \Ab^{\overline{f_{\bm{a}}(\lambda)}_{(\bm{\sigma}, \bm{m})}} \otimes t^{f_{\bm{a}}(\lambda)}
      =\Bb^{f_{\bm{a}}(\lambda)}.
  \]  
  Since $f_{\bm{a}} \left( \lr{\mathbb{Z}^n}{\Bb_{\mathrm{ord}(\bm{\sigma})}} \right)=\lr{\mathbb{Z}^n}{\Bb}$,
  the above isomorphism is indeed a support-isomorphism.
\qed
\begin{prop} \label{prop4}
  Suppose that $\Ab, \Ab'$ are central-simple algebras, and 
  let $\Bb=\Mm$ and $\Bb'=\Mmb$ be multiloop algebras of nullity $n$.
  Then $\Bb \suppisom \Bb'$ if and only if there exist $P \in \mathrm{GL}_n(\mathbb{Z})$
  and an algebra isomorphism $\varphi: \Ab \to \Ab'$ such that
  \begin{equation} \label{2'}
    \bm{\sigma}'=\varphi \bm{\sigma}^P \varphi^{-1}.
  \end{equation}
  (In particular, it does not depend on $\bm{m}$ or $\bm{m}'$ whether or not $\Bb \suppisom \Bb'$).  
  Moreover, if $P \in \mathrm{GL}_n(\mathbb{Z})$ and an isomorphism $\varphi: \Ab \to \Ab'$
  satisfy \textup{(\ref{2'})}, then we can take a support-isomorphism $\psi:\Bb \to \Bb'$ satisfying
  $\psi(x \otimes 1) =\varphi(x) \otimes 1$ for $x \in \Ab^{\bm{\sigma}}$.
\end{prop}
{\textbf{Proof.}}
  First, we show the ``if'' part.
  Let $M=\mathrm{l.c.m} \{ \bm{m}, \bm{m}' \} \in \mathbb{Z}_{>0}$ be the least common multiple 
  of $2n$ positive integers $m_1,\dots,m_n,m_1',\dots,m_n'$, 
  and let $\bm{M}=(M,M,\dots,M) \in \mathbb{Z}^n_{>0}$.
  Obviously, $\bm{\sigma}^{\bm{M}}=\bm{\sigma'}^{\bm{M}}=\bm{\mathrm{id}}$.
  By Lemma \ref{lem10},
  \[ \Bb \suppisom M_{\mathrm{ord}(\bm{\sigma})}(\Ab, \bm{\sigma}) \suppisom M_{\bm{M}}(\Ab, \bm{\sigma}) 
  \]
  and
  \[ \Bb' \suppisom M_{\mathrm{ord}(\bm{\sigma}')}(\Ab', \bm{\sigma}')
     \suppisom M_{\bm{M}}(\Ab', \bm{\sigma}').
  \]
  It is clear from Definition \ref{def4} that $P$ is $(\bm{M}, \bm{M})$-admissible,
  and hence it follows from Proposition \ref{realization theorem} that 
  \[ M_{\bm{M}}(\Ab, \bm{\sigma}) \cong_{\mathrm{ig}} M_{\bm{M}}(\Ab', \bm{\sigma}'),
  \]
  in particular $M_{\bm{M}}(\Ab, \bm{\sigma}) \suppisom M_{\bm{M}}(\Ab', \bm{\sigma}')$ by Lemma \ref{supp and ig}.
  Thus, we have $\Bb \suppisom \Bb'$, and the ``if'' part follows.
  The second statement of the proposition is easily checked from the above proof of ``if'' part,
  using Proposition \ref{realization theorem} and the proof of Lemma \ref{lem10}.
  Next, we show the ``only if'' part.
  By Lemma \ref{Z-span}, there exist $Q,R \in \mathrm{GL}_n(\mathbb{Z})$ such that
  \[ \Big\langle \mathrm{supp}_{\mathbb{Z}^n} \left( M_{\mathrm{ord}(\bm{\sigma}^Q)}(\Ab, \bm{\sigma}^Q) \right) \Big\rangle
     =\mathbb{Z}^n, 
  \]
  and
  \[ \Big\langle \mathrm{supp}_{\mathbb{Z}^n} \left( M_{\mathrm{ord}({\bm{\sigma}'}^R)}(\Ab', \bm{{\sigma}'}^R) \right) 
     \Big\rangle =\mathbb{Z}^n.
  \]
  We abbreviate 
  \[ \Bb_Q= M_{\mathrm{ord}(\bm{\sigma}^Q)}(\Ab, \bm{\sigma}^Q) \ \mathrm{and} \
     \Bb'_R= M_{\mathrm{ord}({\bm{\sigma}'}^R)}(\Ab', {\bm{\sigma}'}^R).
  \]
  From the ``if'' part and the assumption, we have $\Bb_Q \suppisom \Bb \suppisom \Bb' \suppisom \Bb'_R$,
  and this gives $\Bb_Q \igisom \Bb'_R$ by Lemma \ref{supp and ig} (b).
  From Proposition \ref{realization theorem}, there exist $S \in \mathrm{GL}_n(\mathbb{Z})$
  and an algebra isomorphism $\varphi: \Ab \to \Ab'$ 
  such that ${\bm{\sigma}'}^R= \varphi \bm{\sigma}^{QS} \varphi^{-1}$.
  Then we have $\bm{\sigma}'= \varphi \bm{\sigma}^{QSR^{-1}} \varphi^{-1}$.
\qed
\section{Multiloop Lie algebras}
\subsection{Preliminary lemmas} \label{susecition1}
Suppose that $\mathfrak{g}$ is a finite dimensional simple Lie algebra.
Note that $\gn$ is central-simple since $k$ is algebraically closed.
For $n\in \mathbb{Z}_{>0}$, let
$\bm{\sigma}=(\sigma_1,\dots,\sigma_n)\in\mathrm{Aut}_{\mathrm{cfo}}^{n}(\mathfrak{g})$ 
and $\bm{m}=(m_1,\dots,m_n) \in \mathbb{Z}^n_{>0}$ satisfying $\idcond{\sigma}{m}$. 
As (\ref{1}), we define a $\bar{\Lambda}_{\bm{m}}$-grading on $\mathfrak{g}$ as
\begin{equation}\label{1'}
  \gn^{\bar{\lambda}}(=\mathfrak{g}^{\bar{\lambda}_{(\bm{\sigma},\bm{m})}})
  :=\{g\in\mathfrak{g} \mid \sigma_{i}(g)=\zeta_{m_i}^{l_i}g \quad \mathrm{for} \ 1\le i \le n \}
\end{equation}
for $\bar{\lambda}=(\bar{l_1},\dots,\bar{l_n})\in\bar{\Lambda}_{\bm{m}}$.
We denote the Killing form on $\mathfrak{g}$ by $( \ | \ )$.
Recall that the Killing form is non-degenerate on $\gn$, invariant, symmetric, and preserved by any automorphisms.
Then since $( \ | \ )$ is preserved by $\sigma_i$'s, we have that 
\begin{equation} \label{4}
  (\mathfrak{g}^{\bar{\lambda}}|\mathfrak{g}^{\bar{\mu}}) = 0 \ \ \text{if} \ \bar{\lambda} + \bar{\mu} \neq \bar{0},
\end{equation}
where $\bar{\lambda}, \bar{\mu} \in \bar{\Lambda}_{\bm{m}}$. 
Also we have that $( \ | \ )$ on $\mathfrak{g}^{\bar{\lambda}} \times \mathfrak{g}^{-\bar{\lambda}}$ is non-degenerate
since $( \ | \ )$ on $\mathfrak{g}$ is non-degenerate.

The following lemma is well-known. (For example, see \cite[Proposition 4.1.]{MR0068531}).

\begin{lem}\label{lem1}
  $\gsigma(=\mathfrak{g}^{\bar{0}})$ is a reductive Lie algebra.
\end{lem}

\begin{rem}\label{rem1} \normalfont
  Note that it is possible that 
  $\gsigma = \{0\}$.
\end{rem} 

Assume that $\gsigma \neq \{ 0 \} $.
Since $\gsigma$ is reductive, 
we can take (and fix) a Cartan subalgebra (i.e.\ a maximal ad-diagonalizable subalgebra) $\hz$ of $\gsigma$.
Note that $\mathfrak{h}$ is not necessarily a Cartan subalgebra of $\mathfrak{g}$.
\begin{lem}\label{lem2} 
  \begin{enumerate}
    \item[\textup{(a)}] $( \ \mid \ )$ is non-degenerate on $\hz$. 
    \item[\textup{(b)}] $\mathfrak{h}$ is ad-diagonalizable on $\mathfrak{g}$. 
  \end{enumerate}
\end{lem}
{\textbf{Proof.}}
  (a) We have the root space decomposition of $\gsigma$ with respect to $\hz$
      \[ \gsigma=\bigoplus_{\alpha \in \hz^*} \gsigma_\alpha,
      \]
      where $\gsigma_\alpha := 
      \{ g \in \gsigma \mid\ [h,g]= \langle \alpha , h \rangle g \quad \mathrm{for} \ h \in \hz \} $.
      Note that $\gsigma_0=\hz$.
      For $h \in \hz$, $\alpha, \beta \in \hz^*$ and $x \in \gsigma_\alpha, y \in \gsigma_\beta $,
      \[ \langle \alpha, h \rangle (x|y)=([h,x]|y)
        =-(x|[h,y])=-\langle \beta, h \rangle (x|y)
      \]
      since $ ( \ | \ ) $ is invariant. 
      This means that 
      \begin{equation} \label{6}
       (x|y)=0 \ \mathrm{unless} \ \alpha + \beta =0.
      \end{equation}
      Hence (a) follows since $( \ | \ )$ is non-degenerate on $\gsigma$.
      
  (b) For any $h \in \hz$, 
      we denote the Jordan decomposition of $\mathrm{ad}_{\gn}(h) $ by 
      \[ \mathrm{ad}_{\gn}(h) = S+T \quad S, T \in \mathfrak{gl}(\gn),
      \]
      where $S$ is the semisimple part and $T$ is the nilpotent part.
      By \cite[Lemma 4.2.B]{MR499562}, $T$ is a derivation on $\gn$. 
      Hence, there exists some element $h_T \in \gn$ such that $\mathrm{ad}_{\gn}(h_T)=T$ since $\gn$ is simple.
      Due to the property of the Jordan decomposition, there exists a polynomial $f(t) \in k[t]$ such that
      \begin{equation}\label{14}
        T=\mathrm{ad}_{\gn}(h_T)=f(\mathrm{ad}_{\gn}(h)),
      \end{equation}
      and this implies that
      \begin{equation}\label{12}
        \mathrm{ad}_{\gn}(h_T)(\gn^{\bar{\lambda}}) \subseteq \gn^{\bar{\lambda}} \quad 
        \mathrm{for} \ \bar{\lambda} \in \Lamm
      \end{equation}
      since $h \in \gn^{\bm{\sigma}}$.
      Thus, $h_T \in \gsigma$.
      From (\ref{14}), $T|_{\gsigma}=\mathrm{ad}_{\gsigma}(h_T)$ is diagonalizable.
      Then since $\mathrm{ad}_{\gsigma}(h_T)$ is nilpotent,
      we have $\mathrm{ad}_{\gsigma}(h_T)=0$.
      Hence, we have $h_T \in \hz $, which gives that $[z,h_T]=0$ for all $z \in \hz$.
      It follows from this and the nilpotency of $\mathrm{ad}_{\gn}(h_T)$ that
      \begin{equation}\label{13}
        (z|h_T)=\mathrm{Tr}\big(\mathrm{ad}_{\gn}(z)\, \mathrm{ad}_{\gn}(h_T)\big)=0 \quad
        \mathrm{for \ all} \ z \in \hz.
      \end{equation}
      By Lemma \ref{lem2} (a) and (\ref{13}), 
      we have $h_T=0$. 
      Hence $\mathrm{ad}_{\mathfrak{g}}(h)$ is semisimple, and (b) follows.
\qed
\subsection{The definition of multiloop Lie algebras} \label{subsection2}
In section \ref{section1}, we have defined a multiloop algebra based on a general algebra.
By the abuse of language, we use a term ``multiloop Lie algebra'' in a different sense from that.

Suppose that $\mathfrak{g}$ is a finite dimensional simple Lie algebra.
For $n\in \mathbb{Z}_{>0}$, let
$\bm{\sigma}=(\sigma_1,\dots,\sigma_n)\in\mathrm{Aut}_{\mathrm{cfo}}^{n}(\mathfrak{g})$ 
and $\bm{m}=(m_1,\dots,m_n) \in \mathbb{Z}^n_{>0}$ satisfying $\idcond{\sigma}{m}$. 
In the following, we define a subalgebra $\hz \subseteq \gsigma$ and an abelian group $Q_{\hz}$, and then
we define a {\textit{multiloop Lie algebra}} $\Lmh$ as a $Q_{\hz} \times \mathbb{Z}^n$-graded Lie algebra.

First, we assume that $\gsigma \neq \{ 0 \}$.
In this case, we take $\hz$ as a Cartan subalgebra of $\gsigma$.
By Lemma \ref{lem2} (b), we can define the root space decomposition of $\gn$ with respect to $\hz$,
which we denote by 
$\mathfrak{g}=\bigoplus_{\alpha \in \mathfrak{h}^*} \mathfrak{g}_{\alpha}$ 
where $\mathfrak{g}_{\alpha}:= \{g \in \mathfrak{g} \mid [h,g] 
=\langle \alpha, h \rangle g \ \ \mathrm{for} \ h \in \hz \}$. 
Put 
\[ \Delta = \mathrm{supp}_{\mathfrak{h}^*}(\mathfrak{g}) \setminus \{ 0 \} \subseteq \hz^*,
\]
and let 
$Q_{\hz} = \sum_{\alpha \in \Delta} \mathbb{Z} \alpha \subseteq \hz^*$.
This grading, together with the grading defined in (\ref{1'}),
gives a $Q_{\hz} \times \Lamm$-grading on $\gn$ as
\begin{equation} \label{42}
  \gn=\bigoplus_{(\alpha, \bar{\lambda}) \in Q_{\hz} \times \Lamm}\gn_{\alpha}^{\bar{\lambda}},
\end{equation}
where we set $\mathfrak{g}_{\alpha}^{\bar{\lambda}} = \mathfrak{g}_{\alpha} \cap \mathfrak{g}^{\bar{\lambda}}$. 
Then we can define a $Q_{\hz} \times \mathbb{Z}^n$-graded Lie algebra $\Lmh$ as
\[ \Lmh=\bigoplus_{(\alpha, \lambda) \in Q_{\hz} \times \mathbb{Z}^n}
   \gn_{\alpha}^{\bar{\lambda}} \otimes t^{\lambda}.
\]

Next, we assume that $\gsigma= \{ 0 \}$.
For the notational convenience, in this case we let $\hz = \gsigma= \{ 0 \}$ 
and $Q_{\hz}$ be a trivial group, and we define
\[ \Lmh = \bigoplus_{\lambda \in \mathbb{Z}^n} \gn^{\bar{\lambda}} \otimes t^{\lambda}.
\] 
Also in this case, we consider $\Lmh$ as a $Q_{\hz} \times \mathbb{Z}^n(\cong \mathbb{Z}^n)$-graded Lie algebra.

Note that, as a $\mathbb{Z}^n$-graded Lie algebra, $\Lmh = M_{\bm{m}}(\gn, \bm{\sigma})$.
\begin{dfn}\label{def1}  \normalfont
  Suppose that $\gn$ is a finite dimensional simple Lie algebra, $\bm{\sigma} \in \Aut$, 
  and $\bm{m} \in \mathbb{Z}^n_{>0}$ such that $\idcond{\sigma}{m}$. 
  Then we call the $Q_{\hz} \times \mathbb{Z}^n$-graded Lie algebra $\Lmh$ defined above
  the \textit{multiloop Lie algebra} determined by $\gn, \bm{\sigma}, \bm{m}, \hz$. 
  We call the positive integer $n$ the \textit{nullity} of $\Lmh$.   
\end{dfn}
\begin{rem} \normalfont
(a) In the definition of a multiloop algebra $M_{\bm{m}}(\Ab, \bm{\sigma})$,
    $\Ab$ is not supposed to be either finite dimensional or simple.
    Thus, it may be more appropriate to call $\Lmh$ in Definition \ref{def1}
    a multiloop Lie algebra based on a finite dimensional simple Lie algebra.
    In this paper, however, we consider a finite dimensional simple case only.
    Thus, we call it simply a multiloop Lie algebra.
    
(b) Even in the case where $\gsigma \neq \{ 0 \}$, $\Delta= \mathrm{supp}_{Q_{\hz}}(\gn) \setminus \{ 0 \}$
    does not necessarily coincide with the root system of $\gn$ since $\hz$ is not necessarily a Cartan subalgebra of $\gn$.
    It is, however, proved in the next subsection that $\Delta$ is an irreducible (possibly non-reduced) 
    finite root system.
\end{rem}

Henceforth, we consider $\gsigma$ as a Lie subalgebra of $\Lmh$ 
using the isomorphism $\mathfrak{g}^{\bm{\sigma}} \to \mathfrak{g}^{\bm{\sigma}} \otimes 1$.
\subsection{Properties of $\Delta$} \label{subsection3}
Let $\mathfrak{g}$ be a finite dimensional simple Lie algebra, $\bm{\sigma}=(\sigma_1,\dots,\sigma_n) 
\in \Aut$, $\bm{m}=(m_1,\dots,m_n) \in \mathbb{Z}_{>0}^n$ where $\idcond{\bm{\sigma}}{\bm{m}}$,
and suppose that $\gsigma \neq \{ 0 \}$.
We take a Cartan subalgebra $\hz \subseteq \gsigma$, and   
define a $Q_{\hz} \times \Lamm$-grading on $\mathfrak{g}$ as (\ref{42}).
Put $\Delta = \mathrm{supp}_{Q_{\hz}}(\gn) \setminus \{ 0 \}$. 

First, since $\gn$ is finite dimensional, the following lemma is obvious:
\begin{lem}\label{lem3}
  $\Delta$ is a finite set.
\end{lem}

Next, by Lemma \ref{lem2} (a), we can define an isomorphism $\nu: \ \hz \to \hz^*$ canonically by setting
\[ \langle \nu(h), h_1 \rangle = (h|h_1) \quad \mathrm{for} \ h,h_1 \in \hz.
\]
Then we can also define a non-degenerate bilinear form $( \ | \ )$ on $\hz^*$ by setting
\begin{equation} \label{43}
  (\alpha|\beta)=\big(\nu^{-1}(\alpha)|\nu^{-1}(\beta)\big) \ \ \ \ 
  \mathrm{for} \ \alpha , \beta \in \hz^*.
\end{equation}
\begin{lem}\label{lem4}
  The $k$-span of $\Delta$ coincides with $\hz^*$.  
\end{lem}
{\textbf{Proof.}}
  We assume that the $k$-span of $\Delta$ does not coincide with $\hz^*$.
  Then there exists some non-zero element $h \in \hz $ such that $\langle \alpha, h \rangle = 0$
  for all $\alpha \in \Delta$,
  which means that $[h, \gn_\alpha]=0$ for all $\alpha \in \Delta$.
  Hence, we have $[h, \gn ]=0$, which contradicts the simplicity of $\gn$. \qed \\

Let $\alpha \in \Delta$ and $\bar{\lambda} \in \Lamm$ such that 
$\gn_{\alpha}^{\bar{\lambda}} \neq \{ 0 \}$. 
(\ref{4}) and (\ref{6}) imply that $\gn_{-\alpha}^{-\bar{\lambda}} \neq \{ 0 \}$
since $( \ | \ )$ is non-degenerate on $\gn$.
Thus, we can take non-zero elements $x_{\alpha}^{\bar{\lambda}} \in \gn_{\alpha}^{\bar{\lambda}}$ and 
$x_{-\alpha}^{-\bar{\lambda}} \in \gn_{-\alpha}^{-\bar{\lambda}}$.
For $h \in \hz$, we have
\[ (h|[x_{\alpha}^{\bar{\lambda}},x_{-\alpha}^{-\bar{\lambda}} ]) 
   =( [h, x_{\alpha}^{\bar{\lambda}} ]| x_{-\alpha}^{-\bar{\lambda}} )  
   =\langle \alpha, h \rangle (x_{\alpha}^{\bar{\lambda}}|x_{-\alpha}^{-\bar{\lambda}})
   =\big( h|\nu^{-1}(\alpha) \big)(x_{\alpha}^{\bar{\lambda}}|x_{-\alpha}^{-\bar{\lambda}}). 
\]
Thus we have 
\begin{equation}\label{50}
  [\xal, \mxal ] = ( \xal | \mxal ) \nu^{-1}(\alpha) \in \hz
\end{equation}
since $( \ | \ )$ is non-degenerate on $\hz$.
\begin{lem}\label{lem8}
  For $\alpha \in \Delta$, $(\alpha | \alpha) \neq 0$.
\end{lem}
{\textbf{Proof.}}
  For some $\alpha \in \Delta$, we assume that $(\alpha | \alpha)
  = \langle \alpha, \nu^{-1}(\alpha) \rangle=0$.
  We can take $0 \neq \xal \in \gn_{\alpha}^{\bar{\lambda}}$ for some $\bar{\lambda} \in \Lamm$.
  Then there exists some element
  $\mxal \in \gn_{-\alpha}^{-\bar{\lambda}}$ such that 
  $(\xal|\mxal)=1$.
  By (\ref{50}) and the assumption, we can see that the Lie subalgebra of $\gn$
  spanned by $ \{ \nu^{-1}(\alpha), \xal, \mxal \} $, which we denote by $S$, is a three-dimensional nilpotent Lie algebra. 
  Then since $\mathrm{ad}_{\gn}(S)\simeq S$ is also nilpotent (in particular, solvable)
  and $\mathrm{ad}_{\gn}\big(\nu^{-1}(\alpha)\big) \in [\mathrm{ad}_{\gn}(S), \mathrm{ad}_{\gn}(S)]$,
  it follows from the Lie's theorem that $\mathrm{ad}_{\gn}\big(\nu^{-1}(\alpha)\big)$ acts nilpotently on $\gn$. 
  From this and Lemma \ref{lem2} (b), it follows that $\mathrm{ad}_{\gn}\big(\nu^{-1}(\alpha)\big)=0$.
  This forces $\alpha =0$, and this is contradiction since $0 \notin \Delta$.
\qed

Let $\alpha \in \Delta$ and $\bar{\lambda} \in \Lamm$ such that 
$\gn_{\alpha}^{\bar{\lambda}} \neq \{ 0 \}$. 
By Lemma \ref{lem8}, $2 (\alpha|\alpha)^{-1} \in k $ exists.
Thus, we can choose non-zero elements $\xal \in \gn_{\alpha}^{\bar{\lambda}}$ and 
$x_{-\alpha}^{-\bar{\lambda}} \in \gn_{-\alpha}^{-\bar{\lambda}}$ satisfying 
\[ (x_{\alpha}^{\bar{\lambda}}|x_{-\alpha}^{-\bar{\lambda}})=\frac{2}{(\alpha|\alpha)},
\]
and we set 
\begin{equation}\label{24}
  h_{\alpha}=\frac{2\nu^{-1}(\alpha)}{(\alpha|\alpha)} \in \hz.
\end{equation}
Then we have 
\begin{equation}\label{25}
  [ h_{\alpha}, x_{\alpha}^{\bar{\lambda}} ] = 2x_{\alpha}^{\bar{\lambda}}, \ \ 
  [ h_{\alpha}, x_{-\alpha}^{-\bar{\lambda}} ] = -2x_{-\alpha}^{-\bar{\lambda}},
\end{equation}
and using (\ref{50}), 
\begin{equation}\label{51}
  [ x_{\alpha}^{\bar{\lambda}}, x_{-\alpha}^{-\bar{\lambda}}]=h_{\alpha}.
\end{equation}
By (\ref{25}) and (\ref{51}), we can see that the Lie subalgebra of $\gn$ spanned by 
these three elements $ \{x_{\alpha}^{\bar{\lambda}},x_{-\alpha}^{-\bar{\lambda}},h_{\alpha} \}$
is isomorphic to $\mathfrak{sl}_2(k)$.
We call the set of these three elements
a \textit{$\mathfrak{sl}_2(k)$-triple} with respect to $(\alpha, \bar{\lambda})$. 
Note that this set is defined only for the pair $(\alpha, \lam)$
satisfying $\gn_\alpha^{\lam} \neq \{ 0 \}$. 
Also, note that for some $\alpha \in \Delta$ it is possible that 
$h_{\alpha}$ is contained in more than one $\mathfrak{sl}_{2}(k)$-triple. 

For $\alpha \in \Delta$, 
we define a reflection $s_{\alpha}$ on $\hz^*$ by 
\begin{equation}\label{29}
  s_{\alpha}(\gamma)=\gamma - \langle \gamma , h_{\alpha} \rangle \alpha \ \mathrm{for} \ \gamma \in \hz^*.
\end{equation}
\begin{lem}\label{lem5}
  Let $\alpha, \beta \in \Delta$, then
  \begin{enumerate}
    \item[\textup{(a)}]
     $\langle \beta , h_{\alpha} \rangle \in \mathbb{Z}$, 
    \item[\textup{(b)}] $s_{\alpha}(\Delta) = \Delta$.
  \end{enumerate}
\end{lem}  
{\textbf{Proof.}}
  We have some $\bar{\lambda}, \bar{\mu} \in \Lamm$ such that 
  $\gn_{\alpha}^{\bar{\lambda}} \neq \{ 0 \}$ and $\gn_{\beta}^{\bar{\mu}} \neq \{ 0 \}$,
  and by the above construction
  we can take a $\mathfrak{sl}_2(k)$-triple $
  \{x_{\alpha}^{\bar{\lambda}},x_{-\alpha}^{-\bar{\lambda}},h_{\alpha} \}$
  with respect to $(\alpha$, $\bar{\lambda})$.
  Let $S_{\alpha}^{\bar{\lambda}}$ be the subalgebra of $\gn$ spanned by these elements.

  (a) We can consider $\gn$ as a $S_{\alpha}^{\lam}$-module by the adjoint action. 
      Since $\gn_{\beta}^{\bmu}$ is nonzero eigenspace for $h_{\alpha}$,
      (a) follows from the representation theory of $\mathfrak{sl}_2(k)$. 

  (b) It suffices to show that 
      \begin{equation}\label{32}
        s_{\alpha}(\beta) \in \Delta. 
      \end{equation}
      We construct an automorphism of $\gn$ using the elements 
      $\xal$ and $\mxal $.
      Since $\Delta$ is a finite set and 
      \[ \mathrm{ad} (\xal) (\gn_{\gamma}) \subseteq \gn_{\alpha + \gamma}
      \]
      for $\gamma \in \Delta \cup \{ 0 \}$, we can see that $\mathrm{ad}(\xal)$ is nilpotent,
      and so is $\mathrm{ad}(\mxal)$.
      Therefore, 
      \[ \theta_{\alpha}^{\lam} := \mathrm{exp}\big(\mathrm{ad}(\xal)\big) 
        \mathrm{exp}\big(-\mathrm{ad}(\mxal)\big) \mathrm{exp}\big(\mathrm{ad}(\xal)\big) \in \mathrm{Aut}(\gn)
      \]
      is a well-defined automorphism of $\gn$.
      To show (\ref{32}), it suffices to show that
      \[ \theta_{\alpha}^{\bar{\lambda}}(\gn_{\beta}) \subseteq \gn_{s_{\alpha}(\beta)}.
      \]
      Let $x_\beta \in \gn_{\beta}$.
      For $h \in \hz$ such that $\langle \alpha, h \rangle= 0$, using $\auttheta(h) = h$, we have 
      \[ [h, \theta_{\alpha}^{\lam}(x_{\beta})]
        = \theta_{\alpha}^{\lam}([h, x_{\beta}]) 
        = \langle \beta, h \rangle \theta_{\alpha}^{\lam}(x_{\beta})
        = \langle s_{\alpha}(\beta), h \rangle \theta_{\alpha}^{\lam}(x_{\beta}).
      \]
      Thus, we have only to check that 
      \[ [h_{\alpha}, \theta_{\alpha}^{\lam}(x_{\beta})]
        = \langle s_{\alpha}(\beta), h_{\alpha} \rangle \theta_{\alpha}^{\lam}(x_{\beta}).
      \]
      This follows from
      \begin{equation}\label{38}
        \theta_{\alpha}^{\lam}(h_\alpha)= -h_{\alpha}
      \end{equation}
      and
      \[ \langle s_{\alpha}(\beta), h_{\alpha} \rangle 
        = \big\langle \beta -\langle \beta, h_{\alpha} \rangle \alpha ,h_{\alpha} \big\rangle
        =- \langle \beta, h_{\alpha} \rangle        
      \]
      ((\ref{38}) follows from an easy calculation in $\sltwo$).
\qed
\\

Now, we show the following proposition:
\begin{prop}\label{prop1}
  Let $\gn$ be a finite dimensional simple Lie algebra and 
  $\bm{\sigma}=(\sigma_1,\dots,\sigma_n) \in \Aut$ such that $\gsigma \neq \{ 0 \}$,
  and let $\hz$ be a Cartan subalgebra of $\gsigma$.
  Then $\Delta:=\mathrm{supp}_{Q_{\hz}}(\gn) \setminus \{ 0 \}$ 
  is an irreducible (possibly non-reduced) finite root system in $\mathfrak{h}^*$ (cf. {\cite[Chapter IV]{MR1890629}}).
\end{prop}
{\textbf{Proof.}}
  By Lemma \ref{lem3}, \ref{lem4}, (\ref{24}), (\ref{29}), Lemma \ref{lem5} (a) and (b),
  we have that $\Delta$ is a (possibly non-reduced) finite root system. 
  Thus, it suffices to show that $\Delta$ is irreducible.
  We assume that $\Delta=\Delta_1 \cup \Delta_2 , \ ( \Delta_1 |\Delta_2 )=0$
  and $\Delta_1 \neq \emptyset$. 
  Let $\gn(\Delta_1)$ be a subalgebra in $\gn$ generated by 
  $ \cup_{\alpha \in \Delta_1} {\gn_{\alpha}}  $.
  If $\alpha \in \Delta_1$, $\beta \in \Delta_2$,
  we have from Lemma \ref{lem8} that $(\alpha + \beta| \alpha) \neq 0$,
  $(\alpha + \beta| \beta) \neq 0$, and hence we have $\alpha + \beta \notin \Delta$. 
  Thus, since $\alpha + \beta \neq 0$ we have $\gn_{\alpha + \beta} = \{ 0 \}$, and this means
  \begin{equation}\label{41}
    [\gn_\alpha, \gn_{\beta}] ={0}.
  \end{equation}
  Then we can easily see that $\gn(\Delta_1)$ is a nonzero ideal of $\gn$, 
  which coincides with $\gn$.
  Since $[\gn_{\beta}, \gn(\Delta_1)]=0$ for any $\beta \in \Delta_2$ by (\ref{41}), 
  $\Delta_2 = \emptyset$.
\qed \\

Then the following corollary is obvious from the definition of a multiloop 
Lie algebra $\Lmh$.
\begin{cor}\label{cor1}
Let $\gn$, $\hz$, $\bm{\sigma}$ be as in Proposition \ref{prop1} 
(in particular, $\gsigma \neq \{ 0 \}$), and let $\bm{m} \in \mathbb{Z}^n$ satisfy $\idcond{\sigma}{m}$.
Then $\Delta := \mathrm{supp}_{Q_{\hz}} \big( \Lmh \big) \setminus \{ 0 \}$
is an irreducible (possibly non-reduced) finite root system.
\end{cor}
\section{Support-isomorphism of multiloop Lie algebras}
Let $\L =\Lmh, \L'=\Lmhb$ be multiloop Lie algebras of nullity $n$.
As defined in the previous section, $\L$ is $Q_{\hz} \times \mathbb{Z}^n$-graded 
and $\L'$ is $Q_{\hz'} \times \mathbb{Z}^n$-graded.
Thus, $\L$ and $\L'$ are support-isomorphic if and only if there exist a Lie algebra isomorphism
$\varphi: \L \to \L'$ and a group isomorphism $\varphi_{\mathrm{su}}: \lr{Q_{\hz} \times \mathbb{Z}^n}{\L}
\to \lr{Q_{\hz'} \times \mathbb{Z}^n}{\L'}$ such that 
\[ \varphi(\L_{\alpha}^{\lambda})= 
   {\L'}_{\alpha'}^{\lambda'}
\]
for $(\alpha, \lambda) \in \lr{Q_{\hz} \times \mathbb{Z}^n}{\L}$ where we set 
$\varphi_{\mathrm{su}}\big((\alpha, \lambda)\big)=(\alpha', \lambda')$. 
The goal of this section is to give a necessary and sufficient condition for $\L$ and $\L'$ to be 
support-isomorphic.
\subsection{Some isomorphisms}
In section \ref{section1}, we have observed the conditions for two multiloop algebras, which are $\mathbb{Z}^n$-graded,
to be isograded-isomorphic or support-isomorphic.
To apply those results to multiloop Lie algebras, which are $Q_{\hz} \times \mathbb{Z}^n$-graded,
we define the following:
%
\begin{dfn} \label{def of iso and bi} \normalfont
  Let $\L$ and $\L'$ be multiloop Lie algebras of nullity $n$. Note that we can see $\L$ and $\L'$ as $\mathbb{Z}^n$-graded
  Lie algebras by considering only their $\mathbb{Z}^n$-gradings.
  \begin{enumerate}
  \item[(a)] We say $\L$ and $\L'$ are \textit{$\mathbb{Z}^n$-isograded-isomorphic} 
             if $\L$ and $\L'$ are isograded-isomorphic as $\mathbb{Z}^n$-graded Lie algebras.
             In that case we write $\L \cong_{\mathbb{Z}^n-\mathrm{ig}} \L'$. 
  \item[(b)] We say $\L$ and $\L'$ are \textit{$\mathbb{Z}^n$-support-isomorphic} 
             if $\L$ and $\L'$ are support-isomorphic as $\mathbb{Z}^n$-graded Lie algebras.
             In that case we write $\L \cong_{\mathbb{Z}^n-\mathrm{su}} \L'$.
  \end{enumerate}
\end{dfn}
%

The following lemma is immediately follows from Proposition \ref{prop4}:
\begin{lem} \label{supp-isom for Lie}
  Let $\L=\Lmh$ and $\L'=\Lmhb$ be multiloop Lie algebras of nullity $n$.
  Then $\L \zisom \L'$ if and only if there exist $P \in \mathrm{GL}_n(\mathbb{Z})$
  and an algebra isomorphism $\varphi: \gn \to \gn'$ such that 
  $\bm{\sigma}'=\varphi \bm{\sigma}^P \varphi^{-1}$.
\end{lem}

The following proposition, which can be proved in the almost same way used in the proof of \cite[Proposition 2.1.3]{MR2506428},
shows that if two multiloop Lie algebras are $\mathbb{Z}^n$-isograded-isomorphic or $\mathbb{Z}^n$-support-isomorphic,
then we can choose the isomorphism preserving the root grading. In particular, if two multiloop Lie algebras are 
$\mathbb{Z}^n$-isograded-isomorphic (resp.\ $\mathbb{Z}^n$-support-isomorphic), then they are isograded-isomorphic
(resp. support-isomorphic). 
\begin{prop} \label{prop3}
  Let $\mathfrak{L}=\Lmh$ and $\mathfrak{L}'=\Lmhb$ be multiloop Lie algebras 
  of nullity $n$.
  If $\L$ and $\L'$ are $\mathbb{Z}^n$-isograded-isomorphic (resp. $\mathbb{Z}^n$-support-isomorphic),
  then we can choose a $\mathbb{Z}^n$-isograded-isomorphism (resp. $\mathbb{Z}^n$-support-isomorphism) $\varphi$ 
  satisfying the following condition: there exists a group isomorphism $\varphi_Q: Q_{\hz} \to Q_{\hz'}$ satisfying 
  \begin{equation} \label{100}
    \varphi(\L_{\alpha}) = \L'_{\varphi_Q(\alpha)}
  \end{equation}
  for $\alpha \in Q_{\hz}$.
\end{prop}
{\textbf{Proof.}}
  We only show the $\mathbb{Z}^n$-isograded-isomorphic case. (The proof of the other case is the same).
  Let $\varphi': \L \to \L'$ be a $\mathbb{Z}^n$-isograded-isomorphism.
  If $\gsigma = \{ 0 \}$, we have ${\gn'}^{\bm{\sigma}'}= {\L'}^0 = \varphi(\L^0)= \{ 0 \}$.
  In this case both $Q_{\hz}$ and $Q_{\hzb}$ being trivial groups, 
  (\ref{100}) obviously follows if we put $\varphi= \varphi'$.
  Next, suppose $\gsigma \neq \{ 0 \}$.
  If a $\mathbb{Z}^n$-isograded-automorphism $\psi:\L \to \L$ satisfies
  $\varphi' \circ \psi (\hz)=\hz'$, 
  then it is easily checked that $\varphi = \varphi' \circ \psi$ satisfies (\ref{100}) for suitable $\varphi_Q$.
  Thus, we show that there exists $\psi$ satisfying the above condition.
  By Lemma \ref{lem1}, we can write 
  \[ \gsigma= \s_0 \oplus \s_1 \oplus \dots \oplus \s_k
  \]
  where $\s_0$ is a center and $\s_i $ for $1 \le i \le k$ is a simple ideal.
  Also, since $\hz$ and ${\varphi'}^{-1}(\hz')$ are both the Cartan subalgebras of $\gsigma$, 
  we can write 
  \[ \hz=\s_0 \oplus \hz_1 \oplus \dots \oplus \hz_k \ \mathrm{and} \ 
     {\varphi'}^{-1}(\hz')=\s_0 \oplus \hz_1' \oplus \dots \oplus \hz_k' 
  \]
  where $\hz_i,\hz_i'$ are both the Cartan subalgebras of $\s_i$.
  Using the technique in the proof of \cite[Proposition 2.1.3]{MR2506428}, we can take $\mathbb{Z}^n$-isograded-automorphisms
  $\psi_i$ of $\L$ for $1 \le i \le k$ such that $\psi_i(\hz_i)=\hz_i'$ and 
  $\psi_i(g)=g$ for $g \in \s_j$ if $i \neq j$.
  Then $\psi:=\psi_1 \circ \dots \circ \psi_k$ satisfies $\varphi' \circ \psi (\hz) = \hz'$.
\qed
\subsection{Support-isomorphism of multiloop Lie algebras}
\begin{dfn} \label{def5} \normalfont
  Suppose that $Q, \Lambda$ are abelian groups, 
  and $\B$ is a $Q \times \Lambda$-graded Lie algebra. 
  
  (a) Let $\rho : \langle \mathrm{supp}_{\Lambda}(\B) \rangle \to \Lambda$ be a injective group homomorphism.
               We define a new $Q \times \Lambda$-graded Lie algebra $\B_{(\rho)}$ as follows:
               $\B_{(\rho)}=\B$ as a Lie algebra, 
               and the $Q \times \Lambda$-grading on $\B_{(\rho)}$ is given by 
               \begin{equation}\label{52}
                 ({\B_{(\rho)}})^{\lambda}_{\alpha}=
                 \begin{cases}
                   \B^{\rho^{-1}(\lambda)}_{\alpha} 
                   & \text{if $\lambda \in \mathrm{Im} \, \rho$ } \\
                   \{ 0 \}
                   & \text{if $\lambda \notin \mathrm{Im} \, \rho$ }
                 \end{cases}
               \end{equation} 
               for $\alpha \in Q, \lambda \in \Lambda$. 

  (b) Let $s \in \mathrm{Hom}(Q,\Lambda)$ be a group homomorphism from $Q$ to $\Lambda$.
               We define a new $Q \times \Lambda$-graded Lie algebra
               $\B^{(s)}$ as follows: as a Lie algebra, $\B^{(s)}=\B$ 
               and the grading on $\B^{(s)}$ is given by 
               \[ (\B^{(s)})_{\alpha}^{\lambda}=\B_{\alpha}^{\lambda+s(\alpha)}
               \]
               for $\alpha \in Q, \lambda \in \Lambda$.
               ($\B^{(s)}$ was introduced in \cite{MR2506428}, \cite{MR2743759}).
\end{dfn}
\begin{rem} \label{rem3} \normalfont
  It is easily checked that $\mathrm{supp}_{\Lambda}{(\B_{(\rho)})}= 
  \rho \left( \mathrm{supp}_{\Lambda}(\B) \right)$.
  Thus, we have 
  \begin{equation} \label{9'}
    \lr{\Lambda}{\B_{(\rho)}}=\rho \left(\lr{\Lambda}{\B}\right).
  \end{equation}
\end{rem}
\begin{lem} \label{Z-support for Lie}
  Suppose that $\L= \Lmh$ is a multiloop Lie algebra of nullity $n$,
  and suppose that $P \in \mathrm{GL}_n(\mathbb{Z}), \bm{\tilde{m}} \in \mathbb{Z}_{>0}^n$ satisfy 
  $(\bm{\sigma}^P)^{\tilde{\bm{m}}}= \bm{\mathrm{id}}$.
  Then $\gsigma= \gn^{\bm{\sigma}^P}$,
  and there exists some injective homomorphism $\rho: \lr{\mathbb{Z}^n}{\L} \to \mathbb{Z}^n$ such that
  $\L_{(\rho)}$ is $Q_{\hz} \times \mathbb{Z}^n$-graded-isomorphic 
  to $L_{\bm{\tilde{m}}}(\gn, \bm{\sigma}^P, \hz)$.
\end{lem}
{\textbf{Proof.}}
  By the definition of $\bm{\sigma}^P$, $\gsigma \subseteq \gn^{\bm{\sigma}^P}$ is obvious.
  Then, since $(\bm{\sigma}^P)^{P^{-1}}= \bm{\sigma}$, we have $\gsigma = \gn^{\bm{\sigma}^P}$.
  We write $\L' = L_{\bm{\tilde{m}}}(\gn, \bm{\sigma}^P, \hz)$.
  By Proposition \ref{prop4}, we can take a $\mathbb{Z}^n$-support-isomorphism
  $\psi: \L \to \L'$ such that $\psi|_{\gsigma} = \mathrm{id}_{\gsigma}$.
  Then, since $\psi|_{\hz}= \mathrm{id}_{\hz}$, it is easily checked that 
  \[ \psi(\L_{\alpha})=\L'_{\alpha}
  \]
  for $\alpha \in Q_{\hz}$.
  Let $\psi_{\mathrm{su}}: \lr{\mathbb{Z}^n}{\L} \to \lr{\mathbb{Z}^n}{\L'}$ be a group isomorphism
  such that $\psi(\L^{\lambda})={\L'}^{\psi_{\mathrm{su}}(\lambda)}$ for $\lambda \in \lr{\mathbb{Z}^n}{\L}$,
  and $\iota: \lr{\mathbb{Z}^n}{\L'} \to \mathbb{Z}^n$ be the canonical injective homomorphism.
  We show that $\L_{(\iota \circ \psi_{\mathrm{su}})}$ is $Q_{\hz} \times \mathbb{Z}^n$-graded isomorphic to $\L'$.
  Since $\L_{(\iota \circ \psi_{\mathrm{su}})}=\L$ as a Lie algebra,
  we can see $\psi$ as a Lie algebra isomorphism from $\L_{(\iota \circ \psi_{\mathrm{su}})}$ onto $\L'$.
  If $\lambda \in \lr{\mathbb{Z}^n}{\L'}$, then 
  \[ \psi \left(\left(\L_{(\iota \circ \psi_{\mathrm{su}})}\right)_{\alpha}^{\lambda}\right)
     =\psi \left(\L_{\alpha}^{\psi_{\mathrm{su}}^{-1}(\lambda)}\right)={\L'}_{\alpha}^{\lambda}
  \]
  for $\alpha \in Q_{\hz}$.
  Also if $\lambda \notin \lr{\mathbb{Z}^n}{\L'}$,
  \[ \psi\left(\left(\L_{(\iota \circ \psi_{\mathrm{su}})}\right)_{\alpha}^{\lambda}\right)
     =\{ 0 \} ={\L'}_{\alpha}^{\lambda}
  \]
  for $\alpha \in Q_{\hz}$.
  Thus, $\psi$ is indeed a $Q_{\hz} \times \mathbb{Z}^n$-graded-isomorphism.
\qed
\\

For an algebra $\Ab$ and $\bm{\tau}=(\tau_1, \dots, \tau_n),\bm{\sigma}=(\sigma_1, \dots, \sigma_n)
\in \mathrm{Aut}(\Ab)^n$, we write $\bm{\tau}\bm{\sigma}=
(\tau_1 \sigma_1, \dots, \tau_n \sigma_n) \in \mathrm{Aut}(\Ab)^n$.
\begin{lem} \label{lem6} 
  Let $\L=\Lmh$ be a multiloop Lie algebra of nullity $n$ such that $\gsigma \neq \{ 0 \}$,
  and let $s = (s_1, \dots, s_n) \in \mathrm{Hom}(Q_{\hz}, \mathbb{Z}^n)$.
  For $1 \le i \le n$, we define $\tau_i \in \mathrm{Aut}(\gn)$ by 
  \[ \tau_i(x_{\alpha})=\zeta_{m_i}^{-s_i(\alpha)}(x_{\alpha})
  \]
  for $\alpha \in Q_{\hz}, x_{\alpha} \in \gn_{\alpha}$.
  Then 
  \begin{equation} \label{67}
     \bm{\tau}\bm{\sigma} \in \Aut, \ (\bm{\tau}\bm{\sigma})^{\bm{m}}=\bm{\mathrm{id}},
  \end{equation}
  $\hz$ is a Cartan subalgebra of $\gn^{\bm{\tau}\bm{\sigma}}$,
  and $\L^{(s)}$ is $Q_{\hz} \times \mathbb{Z}^n$-graded-isomorphic to $L_{\bm{m}}(\gn, \bm{\tau} \bm{\sigma},  \hz)$. 
\end{lem}
{\textbf{Proof.}}
  It is clear that $\sigma_1, \dots, \sigma_n, \tau_1, \dots, \tau_n$ commute with each other 
  and $\idcond{\tau}{m}$. Then (\ref{67}) is easily checked.
  $\hz \subseteq \gn^{\bm{\tau} \bm{\sigma}}$ is obvious.
  If $g \in \gn^{\bm{\tau} \bm{\sigma}}$ satisfies $[ \hz, g ] = 0$, that is $g \in \gn^{\bm{\tau} \bm{\sigma}} \cap \gn_0$,
  then we have using $\tau_i|_{\gn_0} =\mathrm{id}$ for all $i$ that 
  $g \in \gn^{\bm{\tau} \bm{\sigma}} \cap \gn_0= \gn^{\bm{\sigma}} \cap \gn_0 = \hz$.
  Therefore, $\hz$ is a Cartan subalgebra of $\gn^{\bm{\tau} \bm{\sigma}}$.  
  The rest of the lemma can be proved in exactly the same way as \cite[Lemma 4.2.4]{MR2506428}. 
  \qed

We introduce the following notation: 
if $q \in \mathbb{Q}$ is expressed as $q=a/b$
where $a \in \mathbb{Z}$ and $b \in \mathbb{Z}_{>0}$, then we set 
\[ \zeta^q=\zeta_b^a.
\]
By (\ref{46}), $\zeta^q$ is well-defined.

Now, we show the following theorem:
\begin{thm} \label{thm for supp-isom}
  Let $\L= \Lmh$ and $\L'=\Lmhb$ be multiloop Lie algebras of nullity $n$.
  Then the following statements are equivalent: 
  \begin{enumerate}
    \item[\textup{(a)}] $\L \suppisom \L'$. 
    \item[\textup{(b)}] 
      There exist $s=(s_1, \dots, s_n) \in 
      \mathrm{Hom}(Q_{\hz},\mathbb{Q}^n)$, $P\in \mathrm{GL}_n(\mathbb{Z})$ and a Lie algebra isomorphism
      $\varphi: \gn \to \gn'$ satisfying the following condition:
      if we define $\tau_i \in \mathrm{Aut}(\gn)$ for $1 \le i \le n$ as 
      \begin{equation} \label{68}
        \tau_i(x_{\alpha})=\zeta^{-s_i(\alpha)}x_{\alpha}
      \end{equation}
      for $\alpha \in Q_{\hz}$, $x_{\alpha} \in \gn_{\alpha}$ and $\bm{\tau}=(\tau_1, \dots, \tau_n)$, then
      \[ \bm{\sigma}'= \varphi (\bm{\tau}\bm{\sigma})^P \varphi^{-1}. 
      \]
    \item[\textup{(c)}]
      There exists a finite sequence of $Q_{\hz} \times \mathbb{Z}^n$-graded Lie algebras
      $\L_0, \L_1,\dots, \L_p$ satisfying the following three conditions: 
      \begin{enumerate}
        \item[\textup{(i)}] $\L_0=\L$.
        \item[\textup{(ii)}] $\L_p \cong_{\mathrm{\mathbb{Z}^n-ig}} \L'$. 
        \item[\textup{(iii)}]
          For $1 \le i \le p-1$, $\L_{i+1}$ is either ${\L_{i}}_{(\rho_i)}$ for some injective homomorphism
          $\rho_i: \lr{\mathbb{Z}^n}{\L_i} \to \mathbb{Z}^n$ or ${\L_i}^{(s_i)}$ 
          for some $s_i \in \mathrm{Hom}(Q_{\hz}, \mathbb{Z}^n)$.
      \end{enumerate}   
  \end{enumerate}
\end{thm}
{\textbf{Proof.}}
  ``(a) $\Rightarrow$ (b)'' 
    If $\gsigma = \{ 0 \}$, then $\L \suppisom \L'$ means $\L \cong_{\mathbb{Z}^n \mathrm{-su}} \L'$,
    and (b) follows from Lemma \ref{supp-isom for Lie}.
    Thus, we suppose that $\gsigma \neq \{ 0 \}$, and let $\psi: \L \to \L'$ be a support-isomorphism.
    Then $\psi(\hz)=\hz'$, and if we define $\hat{\psi}: \hz^* \to {\hz'}^*$ as 
    $\langle \hat{\psi}(\alpha), \psi(h) \rangle = \langle \alpha, h \rangle$
    for $\alpha \in \hz^*, h \in \hz$, it is easily checked that $\psi(\L_{\alpha})=\L'_{\hat{\psi}(\alpha)}$
    for $\alpha \in Q_{\hz}$.
    Thus, we can view $\hat{\psi}$ as a group isomorphism from $Q_{\hz}$ onto $Q_{\hz'}$.
    We define a group homomorphism $p: \lr{Q_{\hz} \times \mathbb{Z}^n}{\L} \to \mathbb{Z}^n$ as
    \[ \psi(\L_{\alpha}^{\lambda})={\L'}_{\hat{\psi}(\alpha)}^{p((\alpha, \lambda))}
    \]
    for $(\alpha, \lambda) \in \lr{Q_{\hz} \times \mathbb{Z}^n}{\L}$.
    By Corollary \ref{cor1}, $\Delta:= \mathrm{supp}_{Q_{\hz}}(\L) \setminus \{ 0 \}$ is an irreducible finite root system.
    Let $\Phi$ be a base of $\Delta$, and for each $\alpha \in \Phi$, 
    we take $\lambda_{\alpha} \in \mathbb{Z}^n$ such that $\L_{\alpha}^{\lambda_{\alpha}} \neq \{ 0 \}$.
    Since $\Phi$ is a $\mathbb{Z}$-basis of $Q_{\hz}$, we can take $t=(t_1, \dots, t_n) \in 
    \mathrm{Hom}(Q_{\hz}, \mathbb{Z}^n)$ satisfying $t(\alpha)=\lambda_{\alpha}$ for $\alpha \in \Phi$.
    Then since $(\L^{(t)})_{\alpha}^0 = \L_{\alpha}^{t(\alpha)} \neq \{ 0 \}$ for $\alpha \in \Phi$,
    $(Q_{\hz},0) \subseteq \lr{Q_{\hz} \times \mathbb{Z}^n}{\L^{(t)}}$.
    Thus, we have 
    \[ \big\langle \mathrm{supp}_{Q_{\hz} \times \mathbb{Z}^n}(\L^{(t)}) \big\rangle = \big\{(\alpha, \lambda) \mid
       \alpha \in Q_{\hz}, \lambda \in \big\langle \mathrm{supp}_{\mathbb{Z}^n}(\L^{(t)}) \big\rangle \big\},
    \]
    and then we have 
    \begin{equation} \label{3'}
      \lr{Q_{\hz} \times \mathbb{Z}^n}{\L} = \big\{ \big( \alpha, \lambda + t(\alpha) \big) 
      \mid \alpha \in Q_{\hz}, \lambda \in \lr{\mathbb{Z}^n}{\L^{(t)}} \big\}.
    \end{equation}
    Next, we define $u=(u_1, \dots, u_n) \in \mathrm{Hom}(Q_{\hz'}, \mathbb{Z}^n)$ as 
    \[ u \big(\hat{\psi}(\alpha)\big)= p\big(\alpha, t(\alpha)\big)
    \]
    for $\alpha \in Q_{\hz}$.
    (Note that $\big(\alpha, t(\alpha)\big) \in \langle \mathrm{supp}_{Q \times \mathbb{Z}^n}(\L) \rangle$ from (\ref{3'})). 
    Since $\L = \L^{(t)}$ and $\L' = {\L'}^{(u)}$ as Lie algebras,
    we can consider $\psi$ as a Lie algebra isomorphism from $\L^{(t)}$ onto ${\L'}^{(u)}$.
    Let $\alpha \in Q_{\hz}$ and $\lambda \in \lr{\mathbb{Z}^n}{\L^{(t)}}$. 
    Then $(\alpha, \lambda + t(\alpha)) \in \lr{Q_{\hz} \times \mathbb{Z}^n}{\L}$ by (\ref{3'}). 
    Thus, we have 
    \[ \psi \big( (\L^{(t)})_{\alpha}^{\lambda} \big) = \psi \big(\L_{\alpha}^{\lambda + t(\alpha)}\big)
       = {\L'}_{\hat{\psi}(\alpha)}^{p((\alpha, \lambda + t(\alpha ) ))}
       =\big({\L'}^{(u)}\big)_{\hat{\psi}(\alpha)}^{p\left((0, \lambda)\right)},
    \]
    and this means 
    \[ \psi \big( (\L^{(t)})^{\lambda}\big) = \big( {\L'}^{(u)} \big)^{p\left((0, \lambda)\right)},
    \]
    for $\lambda \in \lr{\mathbb{Z}^n}{\L^{(t)}}$.
    The map $\lambda \mapsto p\big((0,\lambda)\big)$ defined from $\lr{\mathbb{Z}^n}{\L^{(t)}}$ to $\lr{\mathbb{Z}^n}{{\L'}^{(u)}}$ 
    is obviously additive.
    Also we see that this map is a group isomorphism since $\psi|_{\L_0}$, which maps $\L_0^{\lambda}$ 
    for $\lambda \in \langle \mathrm{supp}_{\mathbb{Z}^n}(\L^{(t)}) \rangle$ to ${\L'}_0^{p((0, \lambda))}$, is a Lie algebra isomorphism.
    Hence, the Lie algebra isomorphism $\psi$ is indeed a $\mathbb{Z}^n$-support-isomorphism 
    from $\L^{(t)}$ to ${\L'}^{(u)}$.
    We define $\tilde{\tau}_i \in \mathrm{Aut}(\gn), \tilde{\tau}'_i \in \mathrm{Aut}(\gn')$ 
    for $1 \le i \le n$ as $\tilde{\tau}_i(x_{\alpha}) = \zeta_{m_i}^{-t_i(\alpha)} \, x_\alpha$
    for $\alpha \in Q_{\hz}, x_{\alpha} \in \gn_{\alpha}$, and 
    $\tilde{\tau}'_i(y_{\beta}) = \zeta_{m_i'}^{-u_i(\beta)} \, y_\beta$
    for $\alpha \in Q_{\hz'}, y_{\beta} \in \gn'_{\beta}$.
    By Lemma \ref{lem6}, $\L^{(t)} \cong_{Q_{\hz} \times \mathbb{Z}^n} L_{\bm{m}}(\gn, \bm{\tilde{\tau}}
    \bm{\sigma}, \hz)$ and ${\L'}^{(u)} \cong_{Q_{\hz'} \times \mathbb{Z}^n}
    L_{\bm{m}'}(\gn', \bm{\tilde{\tau}}' \bm{\sigma}', \hz')$ where $\bm{\tilde{\tau}}=(\tilde{\tau}_1,
    \dots, \tilde{\tau}_n)$ and $\bm{\tilde{\tau}}'=(\tilde{\tau}'_1, \dots, \tilde{\tau}'_n)$.
    Therefore, we have
    \[ L_{\bm{m}}(\gn, \bm{\tilde{\tau}} \bm{\sigma}, \hz) \cong_{\mathbb{Z}^n \mathrm{-su}}
       L_{\bm{m}'}(\gn', \bm{\tilde{\tau}}' \bm{\sigma}', \hz'),
    \]
    and then from Lemma \ref{supp-isom for Lie}, there exist $P \in \mathrm{GL}_n (\mathbb{Z})$
    and a Lie algebra isomorphism $\varphi: \gn \to \gn'$ such that
    \begin{equation} \label{4'}
      \bm{\tilde{\tau}}' \bm{\sigma}' = \varphi (\bm{\tilde{\tau}} \bm{\sigma})^P \varphi^{-1}.
    \end{equation}
    Using a similar argument as the proof of Proposition \ref{prop3},
    we can suppose that $\varphi(\hz)=\hz'$.
    Under this assumption we define $\hat{\varphi}: Q_{\hz} \to Q_{\hz'}$
    as $\langle \hat{\varphi}(\alpha), \varphi(h) \rangle = \langle \alpha, h \rangle$. 
    We further set $P^{-1} = (q_{ij})$, and finally 
    we define $s = (s_1, \dots, s_n) \in \mathrm{Hom}(Q_{\hz}, \mathbb{Q}^n)$ as 
    \[ s_j = \frac{1}{m_j}t_j - \sum_{i} \frac{q_{ij}}{m_i'} u_i \circ \hat{\varphi}.
    \]
    If $\tau_i$ is defined by (\ref{68}), we have
    \[  \tau_j(x_{\alpha}) = \left( \prod_{i} \zeta_{m_i'}^{q_{ij}u_i(\hat{\varphi}(\alpha))} \right)
        \cdot \zeta_{m_j}^{-t_j(\alpha)} x_{\alpha} 
        = \varphi^{-1} \circ \left(\prod_i {\tilde{\tau}'}_i{}^{-q_{ij}} \right) \circ \varphi \circ \tilde{\tau}_j 
        (x_{\alpha}) \\
    \]
    for $\alpha \in Q_{\hz}, x_{\alpha} \in \gn_{\alpha}$.
    Thus, $\bm{\tau}= \left( \varphi^{-1} \bm{\tilde{\tau}}'{}^{-P^{-1}} \varphi \right) \bm{\tilde{\tau}}$.
    Then we have from (\ref{4'}) that 
    \[ \begin{split}
         \bm{\sigma}' &= \bm{\tilde{\tau}}'{}^{-1} \big(\varphi (\bm{\tilde{\tau}} \bm{\sigma})^P \varphi^{-1} \big)
         =\varphi ( \varphi^{-1} \bm{\tilde{\tau}}'{}^{-1} \varphi) (\bm{\tilde{\tau}} \bm{\sigma})^P \varphi^{-1} \\
         &=\varphi \left( \left(\varphi^{-1} \bm{\tilde{\tau}}'{}^{-P^{-1}} \varphi \right) \bm{\tilde{\tau}} \bm{\sigma}
         \right)^P \varphi^{-1} =\varphi (\bm{\tau} \bm{\sigma})^P \varphi^{-1},
       \end{split}
    \]
    and (b) follows.
    
    ``(b) $\Rightarrow$ (c)''
      Suppose that $s \in \mathrm{Hom}(Q_{\hz}, \mathbb{Q}^n), P \in \mathrm{GL}_n(\mathbb{Z})$ and 
      $\varphi$ satisfy (b).
      For $1 \le i \le n$, let $a_i \in \mathbb{Z}_{>0}$ be a positive integer satisfying 
      \[ a_i s_i(\alpha) \in \mathbb{Z} \quad \mathrm{for \ all} \ \alpha \in Q_{\hz},
      \]
      and let $\bm{\tilde{m}}=(a_1m_1, \dots, a_nm_n) \in \mathbb{Z}^n_{>0}$.
      From Lemma \ref{Z-support for Lie}, there exists a injective homomorphism $\rho_1: \lr{\mathbb{Z}^n}{\L} \to 
      \mathbb{Z}^n$ such that 
      \begin{equation} \label{5'}
        \L_{(\rho_1)} \cong_{Q_{\hz} \times \mathbb{Z}^n} L_{\bm{\tilde{m}}} (\gn, \bm{\sigma}, \hz).
      \end{equation}
      If we set $t=(a_1m_1s_1, \dots, a_n m_n s_n)$, we have using Lemma \ref{lem6} that
      \begin{equation} \label{6'}
        L_{\bm{\tilde{m}}}(\gn, \bm{\sigma}, \hz)^{(t)} \cong_{Q_{\hz} \times \mathbb{Z}^n}
        L_{\bm{\tilde{m}}}(\gn, \bm{\tau} \bm{\sigma}, \hz).
      \end{equation}
      Using Lemma \ref{Z-support for Lie} again, there exists
      $\rho_2: \lr{\mathbb{Z}^n}{L_{\bm{\tilde{m}}}(\gn, \bm{\tau} \bm{\sigma}, \hz)} 
      \to \mathbb{Z}^n$ such that 
      \begin{equation} \label{7'}
        L_{\bm{\tilde{m}}}(\gn, \bm{\tau} \bm{\sigma}, \hz)_{(\rho_2)}
        \cong_{Q_{\hz} \times \mathbb{Z}^n} L_{\bm{m}'}(\gn, (\bm{\tau} \bm{\sigma})^P, \hz).
      \end{equation} 
      By the assumptions and the definition of a multiloop Lie algebra, it is easily seen that 
      \begin{equation} \label{8'}
        \L' \cong_{\mathbb{Z}-\mathrm{ig}} 
        L_{\bm{m}'}(\gn, (\bm{\tau} \bm{\sigma})^P, \hz).
      \end{equation}
      Then from (\ref{5'}), (\ref{6'}), (\ref{7'}) and (\ref{8'}),
      $\L_0=\L, \L_1=\L_{(\rho_1)}, \L_2={\L_1}^{(t)}, \L_3={\L_2}_{(\rho_2)}$
      is the finite sequence satisfying (c).
      
    ``(c) $\Rightarrow$ (a)'' 
      Suppose that the sequence $\L_0, \L_1, \dots, \L_p$ satisfies (c).
      Then $\L_p \igisom \L'$ by Proposition \ref{prop3}, and we have
      $\L_p \suppisom \L'$ from Lemma \ref{supp and ig} (a).
      Thus, it suffices to show that $\L \suppisom \L_{(\rho)}$ for a injective homomorphism
      $\rho: \lr{\mathbb{Z}^n}{\L} \to \mathbb{Z}^n$, and $\L \suppisom \L^{(s)}$ for
      $s \in \mathrm{Hom}(Q_{\hz}, \mathbb{Z}^n)$.
      The first statement is proved as follows.
      Since $\L=\L_{(\rho)}$ as a Lie algebra, the identity on $\L$ induces a Lie algebra isomorphism
      from $\L$ onto $\L_{(\rho)}$.
      Then since this isomorphism sends $\L_{\alpha}^{\lambda}$ to $(\L_{(\rho)})_{\alpha}^{\rho(\lambda)}$
      for $(\alpha, \lambda) \in \lr{Q_{\hz} \times \mathbb{Z}^n}{\L}$,
      this isomorphism is indeed a support-isomorphism.
      To show the second statement, we consider a Lie algebra isomorphism $\L \to \L^{(s)}$ induced by the identity on $\L$.
      This isomorphism sends $\L_{\alpha}^{\lambda}$ to $(\L^{(s)})_{\alpha}^{\lambda-s(\alpha)}$
      for $\alpha \in Q_{\hz}, \lambda \in \mathbb{Z}^n$.
      Since the map 
      \[ Q_{\hz} \times \mathbb{Z}^n \ni (\alpha, \lambda) \mapsto (\alpha, \lambda-s(\alpha))
         \in Q_{\hz} \times \mathbb{Z}^n
      \]
      is a group isomorphism, this isomorphism is indeed an isograded-isomorphism.
      Then $\L \suppisom \L^{(s)}$ by Lemma \ref{supp and ig} (a).
\qed

\section{The relation among multiloop Lie algebras, Lie tori and extended affine Lie algebras}    

E.\ Neher has introduced in \cite{MR2083842} to construct EALAs from Lie tori.
In this chapter, we consider the construction of EALAs from multiloop Lie algebras that are not 
necessarily Lie tori.
\subsection{Lie $\mathbb{Z}^n$-tori}
In this subsection, using the results of \cite{MR2506428} we give a necessary and sufficient condition 
for a multiloop Lie algebra to be support-isomorphic to some Lie $\mathbb{Z}^n$-torus,
which is defined to be a $Q \times \mathbb{Z}^n$-graded Lie algebra for some root lattice $Q$ satisfying several axioms.
Since it is not needed for the purpose of this paper, we do not state the definition of a Lie $\mathbb{Z}^n$-torus.
(For the definition, see \cite[Definition 1.1.6]{MR2506428}). 

  If $\Delta$ is an irreducible finite root system, we define an \textit{indivisible} root system $\Delta_{\mathrm{ind}}$
  and an \textit{enlarged} root system $\Delta_{\mathrm{en}}$ as
\[ \Delta_{\mathrm{ind}} = \{ \alpha \in \Delta \mid \frac{1}{2} \alpha \notin \Delta \},
\]
and
\[ \Delta_{\mathrm{en}}=\begin{cases}
                          \Delta \cup \{2\alpha\mid \alpha: \mathrm{\ short \ root \ of} \ \Delta \}
                                        & \text{if $\Delta$ has type $B_l, l \ge 1;$} \\
                          \Delta & \text{otherwise.}
                        \end{cases}
\]
By \cite[Proposition 3.2.5]{MR2506428}, we have the following proposition:
%
\begin{prop} \label{condition of multi-Lie}
  Let $\L=\Lmh$ be a multiloop Lie algebra.
  Then $\L$ is a Lie $\mathbb{Z}^n$-torus if and only if 
  the following conditions (A0)-(A3) are satisfied:
  \begin{enumerate}
    \item[\textup{(A0)}] $\bm{m}=\mathrm{ord}(\bm{\sigma})$. 
    \item[\textup{(A1)}] $\gsigma$ is a simple Lie algebra. 
    \item[\textup{(A2)}] If $\bar{0} \neq \bar{\lambda} \in \mathrm{supp}_{\Lamm}(\gn)$,
                then $\gn^{\bar{\lambda}} \cong U^{\bar{\lambda}} \oplus V^{\bar{\lambda}}$
                as a $\gsigma$-module, where $\gsigma$ acts trivially on $U^{\bar{\lambda}}$ and
                either $V^{\bar{\lambda}}=\{ 0 \}$ or $V^{\bar{\lambda}}$ is irreducible of dimension $>1$
                and the weights of $V^{\bar{\lambda}}$ relative to $\hz$ are contained 
                in $(\Delta)_{\mathrm{en}}\cup \{0 \}$
                where $\Delta$ is a root system of $\gn$ relative to $\hz$.
    \item[\textup{(A3)}] $| \langle \{ \sigma_1, \dots, \sigma_n \} \rangle | 
                = \prod_{1 \le i \le n} \mathrm{ord}(\sigma_i)$. 
  \end{enumerate}
\textup{If $\L$ satisfies the conditions (A0)-(A3),
we call $\L$ a \textit{multiloop Lie $\mathbb{Z}^n$-torus} determined by $\gn, \bm{\sigma}, \hz$}.
\end{prop}  
%

Later, we use the following simple lemmas about a finite dimensional simple Lie algebra.
\begin{lem} \label{lem of fin dim simple1}
  Let $\gn$ be a finite dimensional simple Lie algebra, 
  $\bm{\sigma} \in \Aut$ and $\bm{m} \in \mathbb{Z}_{>0}^n$ such that $\idcond{\sigma}{m}$ 
  and $\gsigma \neq \{ 0 \}$,
  and let $\hz$ be a Cartan subalgebra of $\gsigma$.
  We define the $Q_{\hz} \times \Lamm$-grading on $\gn$ as in \textup{(\ref{1})}.
  Suppose that $\alpha \in \Delta:= \mathrm{supp}_{Q_{\hz}}(\gn)\setminus \{ 0 \}$ and $\bar{\lambda} \in \Lamm$
  satisfy $\gn_{\alpha}^{\bar{\lambda}} \neq \{ 0 \}$. Then 
  \begin{enumerate}
    \item[\textup{(a)}] $\gn_{2\alpha}^{2\bar{\lambda}}= \{ 0 \}$, 
    \item[\textup{(b)}] $\dim \gn_{\alpha}^{\bar{\lambda}}=1$.
  \end{enumerate}
\end{lem}
{\textbf{Proof.}}
  Let $\{ \xal, \mxal, h_{\alpha} \}$ be a $\mathfrak{sl}_2(k)$-triple 
  with respect to $(\alpha, \bar{\lambda})$, and we denote by $S_{\alpha}^{\bar{\lambda}}$ the subalgebra spanned 
  by these elements.
  Recall that $[\gn_{\alpha}^{\bar{\lambda}}, \gn_{-\alpha}^{-\bar{\lambda}}] \subseteq k h_{\alpha}$
  (cf. subsection \ref{subsection3}).\\
  (a) Suppose that $\gn_{2\alpha}^{\bar{2\lambda}} \neq \{ 0 \}$ and take $0 \neq z 
      \in \gn_{2\alpha}^{\bar{2\lambda}}$. 
      Note that $\Delta$ is a irreducible finite root system by Proposition \ref{prop1}.
      Since $\mathrm{ad}(\xal)(z)= \{ 0 \}$ and $z$ is an eigenvector for $\mathrm{ad}(h_{\alpha})$
      with eigenvalue 4,
      \[ V:= \sum_{0 \le i}\big(\mathrm{ad}(S_{\alpha}^{\bar{\lambda}})\big)^i(z)
      \]
      is a 5-dimensional irreducible $S_{\alpha}^{\bar{\lambda}}$-module.
      On the other hand, since 
      \[ \big(\mathrm{ad}(\mxal)\big)^2(z) \subseteq [\mxal, \gn_{\alpha}^{\bar{\lambda}}] \subseteq k h_{\alpha},
      \]
      $V$ contains $h_{\alpha}$.
      Then we have $S_{\alpha}^{\bar{\lambda}} \subseteq V$, and this contradicts the irreducibility of $V$.\\
  (b) For $w \in \gn_{\alpha}^{\bar{\lambda}}$, $[\xal, w]=0$ by (a).
      If $[\mxal,w]=0$, then we have 
      \[ [h_{\alpha},w]=[[\xal, \mxal], w]=0,
      \]
      and this implies $w=0$.
      Therefore, $\mathrm{ad}(\mxal)$ is an injective $k$-linear map from $\gn_{\alpha}^{\bar{\lambda}}$
      to 1-dimensional space $k h_{\alpha}$, thus (b) follows.
\qed
\\
 
The following lemma follows from \cite[Lemma 3.2.4]{MR2506428}:
\begin{lem} \label{lem of fin simple2}
  Let $\gn$ be a finite dimensional simple Lie algebra and $W$ is a finite dimensional $\gn$-module.
  We set $\Delta$ be a root system of $\gn$ relative to a Cartan subalgebra $\hz$.
  If the weights of $W$ relative to $\hz$ are contained 
  in $\Delta_{\mathrm{en}}\cup \{ 0 \}$ and $\dim W_{\alpha} \le 1$
  for $\alpha \in \Delta_{\mathrm{en}}$, then $W=U\oplus V$ where $\gn$ acts trivially on $U$ and either
  $V=\{ 0 \}$ or $V$ is irreducible of dimension $>1$.
\end{lem}
%
\begin{thm} \label{Main theorem}
  Let $\L=\Lmh$ be a multiloop Lie algebra of nullity $n$.
  Then $\L$ is support-isomorphic to some multiloop Lie $\mathbb{Z}^n$-torus 
  if and only if $\gsigma \neq \{ 0 \}$.
\end{thm}
{\textbf{Proof.}} 
  First, we show the ``only if'' part.
  Suppose that $\L \suppisom \Lb$ for a multiloop Lie $\mathbb{Z}^n$-torus $\Lb$.
  Then $\hz = \L_0^0 \cong \Lb_0^0 \neq \{ 0 \}$.
  Thus, $\gsigma \neq \{ 0 \}$ follows.
  Next, we show the ``if'' part. 
  Suppose that $\gsigma \neq \{ 0 \}$. Let $\Delta = \mathrm{supp}_{\hz^*}(\L) \setminus \{ 0 \}$ and 
  $Q_{\hz} = \sum_{\alpha \in \Delta} \mathbb{Z} \alpha$.
  By Corollary \ref{cor1}, $\Delta$ is an irreducible finite root system in $\hz^*$.
  Take an arbitrary base $\Phi$ of $\Delta$ and choose $\lambda_{\alpha} \in \mathbb{Z}^n$ 
  for each $\alpha \in \Phi$ such that $\gn_{\alpha}^{\overline{\lambda_{\alpha}}_{(\bm{\sigma}, \bm{m})}} 
  \neq \{ 0 \}$.
  Since $\Phi$ is a $\mathbb{Z}$-basis of $Q_{\hz}$, 
  we can take $s=(s_1, \dots, s_n) \in \mathrm{Hom}(Q_{\hz}, \mathbb{Z}^n)$ such that 
  \[s(\alpha)=\lambda_{\alpha} \quad \mathrm{for} \ \alpha \in \Phi.
  \]
  We define $\tau_i \in \mathrm{Aut}(\gn)$ for $1 \le i \le n$ as
  \[ \tau_i(x_{\alpha})=\zeta_{m_i}^{-s_i(\alpha)}x_{\alpha}
  \]
  for $\alpha \in Q_{\hz}, x_{\alpha} \in \gn_{\alpha}$. 
  Let $\bm{\tau} = (\tau_1, \dots, \tau_n)$ and $\bm{\tilde{\sigma}}=\bm{\tau}\bm{\sigma}$.
  Then $\L^{(s)}$ is $Q_{\hz} \times \mathbb{Z}^n$-graded-isomorphic to 
  $L_{\bm{m}}(\gn, \bm{\tilde{\sigma}}, \hz)$ by Lemma \ref{lem6}.
  Then
  \[ \gn^{\bm{\tilde{\sigma}}}_{\alpha} \cong (\L^{(s)})_{\alpha}^0= \L_{\alpha}^{s(\alpha)} 
     \cong \gn_{\alpha}^{\overline{s(\alpha)}_{(\bm{\sigma}, \bm{m})}}
  \]
  for $\alpha \in \Delta$.
  Thus, we have that 
  \[ \pm \Phi \subseteq \mathrm{supp}_{Q_{\hz}}(\gn^{\bm{\tilde{\sigma}}}) \setminus \{ 0 \} \subseteq \Delta
  \]
  by the construction of $s$.
  Recall that $\gn^{\bm{\tilde{\sigma}}}$ is reductive.
  From the above we see that $\mathrm{supp}_{Q_{\hz}} (\gn^{\bm{\tilde{\sigma}}}) \setminus \{ 0 \}$ spans $\hz^*$.
  This means that $\gn^{\bm{\tilde{\sigma}}}$ has no center, 
  and we see that $\gn^{\bm{\tilde{\sigma}}}$ is a simple Lie algebra with the root system $\Delta_{\mathrm{ind}}$.
  Using \cite[Proposition 5.1.3]{MR2506428}, we can take $P \in \mathrm{GL}_n(\mathbb{Z})$ such that
  \[ | \langle \{ \tilde{\sigma}_1, \dots, \tilde{\sigma}_n \} \rangle |
     = \prod_{1 \le i \le n} \mathrm{ord} \big((\bm{\tilde{\sigma}}^P)_i\big)
  \]
  where we set $\bm{\tilde{\sigma}}^P=\big((\bm{\tilde{\sigma}}^P)_1, \dots, (\bm{\tilde{\sigma}}^P)_n\big)$.
  We prove that $\L':=L_{\mathrm{ord}(\bm{\tilde{\sigma}}^P)}(\gn, \bm{\tilde{\sigma}}^P, \hz)$ 
  is a multiloop Lie $\mathbb{Z}^n$-torus, that is, $\L'$ satisfies condition (A0)-(A3) 
  in Proposition \ref{condition of multi-Lie}.
  (A0) and (A3) are trivial and (A1) has been already shown.
  Since the weights of $\gn$ relative to $\hz$ are contained in $\Delta \cup \{0 \} \subseteq
  (\Delta_{\mathrm{ind}})_{\mathrm{en}} \cup \{ 0 \}$ 
  and $\dim \gn_{\alpha}^{\bar{\lambda}_{(\bm{\tilde{\sigma}}, \bm{m})}} \le 1$  
  for $\alpha \in \Delta, \bar{\lambda} \in \Lamm$ by Lemma \ref{lem of fin dim simple1} (b),
  (A2) follows from Lemma \ref{lem of fin simple2}.
  Thus, $\L'$ is a multiloop Lie $\mathbb{Z}^n$-torus.
  Finally, $\L \suppisom \L'$ follows from Theorem \ref{thm for supp-isom}.
\qed
\subsection{Extended affine Lie algebras}
In this subsection, we consider the construction of an extended affine Lie algebra (EALA, for short)
from a multiloop Lie algebra.

First, we recall the definition of an EALA.
(The following version of the definition is introduced in \cite{MR2083842}).
\begin{dfn} \label{EALA} \normalfont
  An \textit{extended affine Lie algebra} over $k$ of nullity $n$ is a triple $(E, H, ( \ | \ ))$,
  where $E$ is a Lie algebra over $k$, $H$ is a subalgebra of $E$, and $( \ | \ )$ is a bilinear form on $E$, 
  satisfying the following conditions 
  (EA1)-(EA6): 
  \begin{enumerate}
    \item[(EA1)] $( \ | \ )$ is a non-degenerate invariant symmetric bilinear form. 
    \item[(EA2)] $H$ is a nontrivial finite-dimensional self-centralizing and ad-diagonalizable 
                 subalgebra of $E$. 
  \end{enumerate}
  Let $R = \mathrm{supp}_{H^*}(E)$ where we consider a root space decomposition of $E$ with respect to $H$.
  We define $( \ | \ )$ on $H^*$ in a similar way as (\ref{43}) 
  and let $R^0= \{ \alpha \in R \mid (\alpha|\alpha)=0 \}$.  
  \begin{enumerate} 
    \item[(EA3)] For $\alpha \in R \setminus R^0$ and $x_\alpha \in E_{\alpha}$, 
                 $\mathrm{ad}(x_\alpha)$ is locally nilpotent. 
    \item[(EA4)] $R \setminus R^0$ is irreducible. 
    \item[(EA5)] If $E_c$ is the subalgebra in $E$ generated by 
                 $\{ E_{\alpha} \mid \alpha \in R \setminus R^0 \}$, 
                 then $\{ e \in E \mid [e, E_c] = 0 \} \subseteq E_c$.
    \item[(EA6)] $\langle R^0 \rangle \subseteq H^*$ is a free abelian group of rank $n$. 
  \end{enumerate}
  If $(E, H, ( \ | \ ))$ is an EALA, we also say that $E$ is an EALA for short.
\end{dfn}

The following definition is introduced in \cite{MR2743759}:
\begin{dfn} \label{isom of EALA} \normalfont
  Suppose that $(E, H, ( \ | \ ))$ and $(E', H', ( \ | \ )')$ are EALAs.
  We say $(E, H, ( \ | \ ))$ and $(E', H', ( \ | \ )')$ are \textit{isomorphic} 
  if there exists a Lie algebra isomorphism $\chi :E \to E'$ such that
  \[ \chi(H)=H' \ \mathrm{and} \ (\chi(x)|\chi(y))'=a(x|y) \ \mathrm{for \ some} \ a \in k.
  \]
  Alternatively, in that case we say $E$ and $E'$ are \textit{isomorphic as EALAs}.
\end{dfn}

In \cite{MR2083842}, E.\ Neher introduced a construction of a family of EALAs from a Lie $\Lambda$-torus
where $\Lambda$ is a free abelian group of finite rank, 
and he announced that any EALA is constructed in this way.
Observing that construction, we can see that it can be applied to some Lie algebras that are not Lie tori.
Indeed, we show in Proposition \ref{construction of an EALA} that 
if a Lie algebra $\L$ with subalgebra $\hz$ and a bilinear form $( \ | \ )$ satisfies 
the following conditions (L1)-(L4),
we can construct an EALA from $\L, \mathfrak{h}$, and $( \ | \ )$ using Neher's construction: 
\begin{enumerate}
  \item[(L1)] $\L=\oplus_{\lambda \in \Lambda} \L^{\lambda}$ is a graded-central-simple $\Lambda$-graded Lie
             algebra where $\Lambda$ is a free abelian group of finite rank $n$.
  \item[(L2)] Let $\Gamma \subseteq \Lambda$ be a central grading group of $\L$ (cf. Definition \ref{def of graded-simple} (b)). 
              Then the rank of $\Gamma$ is $n$.
  \item[(L3)] $(\ |\ )$ is a non-degenerate invariant symmetric $\Lambda$-graded bilinear form . 
             ($(\ |\ )$ is $\Lambda$-graded means that $(x|y)=0$ for $x \in \L^{\lambda}, y \in \L^{\mu}$
             if $\lambda+\mu \neq 0$).
  \item[(L4)] $0 \neq \hz \subseteq \L^0$, 
             $\hz$ is abelian, ad-diagonalizable on $\L$ and self-centralizing in $\L^0$. 
             Also we assume that $\Delta:=\mathrm{supp}_{\hz^*}(\L) \setminus \{ 0 \}$
             is an irreducible finite root system in $\hz^*$ where we consider a root space decomposition of $\L$ 
             with respect to $\hz$. 
\end{enumerate}

First, we roughly describe this construction for later use.
(For a more precise description, see \cite{MR2083842} or \cite{MR2743759}).
\begin{con} \label{con} \normalfont
  Suppose that $\L$, $\hz$ and $( \ | \ )$ satisfy (L1)-(L4).
  Let $Q_{\hz} = \sum_{\alpha \in \Delta} \mathbb{Z} \alpha$.
  As multiloop Lie algebras, we consider $\L$ as a $Q_{\hz} \times \Lambda$-graded Lie algebra, 
  and write $\L = \bigoplus_{(\alpha, \lambda) \in Q_{\hz} \times \Lambda}\L_{\alpha}^{\lambda}$.
  Let $C(\L)$ be a centroid of $\L$.
  By (L1) and \cite[Lemma 4.3.5 and 4.3.8]{MR2418198}, $C(\L) \cong k[\Gamma]$ as a $\Gamma$-graded algebra
  where $k[\Gamma]$ is a group algebra of $\Gamma$ over $k$.
  Using this isomorphism, we write $C(\L)=\oplus_{\mu \in \Gamma} \, kt^{\mu}$
  where $t^{\mu_1} \cdot t^{\mu_2}= t^{\mu_1+ \mu_2}$ for $\mu_1, \mu_2 \in \Gamma$. 
  For $\theta \in \mathrm{Hom}(\Lambda, k)$, we define a degree derivation $\partial_{\theta}$ of $\L$ by
  \[ \partial_{\theta}(x^{\lambda})=\theta (\lambda)x^{\lambda}
  \]
  for $\lambda \in \Lambda, x^{\lambda} \in \L^{\lambda}$.
  We put 
  \[ \mathrm{CDer}(\L) = C(\L) \cdot \{ \partial_{\theta} \mid \theta \in \text{Hom}(\Lambda, k) \},
  \]
  and
  \[ \mathrm{SCDer}(\L)= \bigoplus_{\mu \in \Gamma} \{ t^{\mu}\cdot 
     \partial_{\theta}\mid \theta(\mu)=0 \}.
  \]
  Then $\mathrm{CDer}(\L)$ is a Lie subalgebra of $\mathrm{Der}(\L)$, and $\mathrm{SCDer}(\L)$ is a Lie subalgebra of $\mathrm{CDer}(\L)$.
  Note that, if $d \in \mathrm{SCDer}(\L)$, we have
  \begin{equation} \label{equ: skew-invariant}
    (d(x)|y) = -(x|d(y))
  \end{equation}
  for $x , y \in \L$ since $( \ | \ )$ is $\Lambda$-graded.
  To construct an EALA from $\L$, we need the following two ingredients:
  \begin{enumerate}
  \item[(i)] Let $\D=\oplus_{\mu \in \Gamma}\D^{\mu}$ be a $\Gamma$-graded subalgebra of $\mathrm{SCDer}(\L)$
             such that the evaluation map $\mathrm{ev}:\Lambda \to (\D^0)^*$ defined by
             $\mathrm{ev}(\lambda)(\partial_{\theta})=\theta(\lambda)$ is injective. 
             Let $\C=\oplus_{\mu \in \Gamma}(\D^{\mu})^*$ 
             and consider $\C$ as a $\D$-module by a contragredient action. 
             We give $\C$ a $\Gamma$-grading by $\C^{\mu}=(\D^{-\mu})^*$. 
  \item[(ii)] Let $\tau: \D \times \D \to \C$ be a 2-cocycle which is graded and invariant, i.e. 
              \[ \tau(\D^{\mu_1}, \D^{\mu_2}) \subseteq \C^{\mu_1+\mu_2}, \quad
                 \tau(d_1, d_2)(d_3)=\tau(d_2, d_3)(d_1) \quad \mathrm{for} \ d_i \in \D,
              \] 
              and we assume that $\tau(\D, \D^0)=0$.
  \end{enumerate}
  Then $E(\L, \D, \tau):= \L \oplus \C \oplus \D$ is a Lie algebra with respect to the product
  \[ \begin{split}
       [x_1 + c_1 + d_1, x_2 + c_2 + d_2]& = \big([x_1, x_2]_{\L} + d_1(x_2) - d_2(x_1)\big) \\
       & +\big(\sigma_{\D} (x_1,x_2) + d_1 \cdot c_2 - d_2 \cdot c_1 + \tau(d_1, d_2)\big) \\
       & +[d_1, d_2]
     \end{split}
  \] 
  for $x_i \in \L, c_i \in \C, d_i \in \D$ 
  where $[ \ , \ ]_{\L}$ denote the product in $\L$, and $\sigma_{\D}: \L \times \L \to \C$ is defined by $\sigma_{\D}(x, y)(d)=(d(x)|y)$
  for $x, y \in \L, d \in \D$. 
  We can define a bilinear form $( \ | \ )$ on $E(\L, \D, \tau)$ by
  \[ (x_1 + c_1 + d_1| x_2 + c_2 + d_2) = (x_1|x_2) + c_1(d_2) + c_2(d_1).
  \]                                     
\end{con}
Then we have the following proposition:
\begin{prop} \label{construction of an EALA}
  If $\L$, $\hz$ and $( \ | \ )$ satisfy the condition \textup{(L1)}-\textup{(L4)}, 
  then $E = E(\L, \D, \tau)$ constructed in Construction \ref{con} is an EALA of nullity $n$
  with respect to the form $( \ | \ )$ and the subalgebra $H = \hz \oplus \C^0 \oplus \D^0$.
\end{prop}
{\textbf{Proof.}}
  We check that $E$, $H$ and $( \ | \ )$ satisfy (EA1)-(EA6).
  $( \ | \ )$ is trivially non-degenerate and symmetric, and using (\ref{equ: skew-invariant}) we can show directly 
  that $( \ | \ )$ is invariant. Hence (EA1) follows.
  Suppose that $x \in E$ satisfies $[x, H]=0$. Then in particular, $[x, \D^0] = 0$. 
  For $\partial_{\theta} \in \D^0, y \in E^{\lambda}, \lambda \in \Lambda$, we have from definition that 
  \[ [\partial_{\theta}, y] =\theta(\lambda)y = \mathrm{ev}(\lambda)(\partial_{\theta})y.
  \]
  Thus, since the evaluation map $\mathrm{ev}: \Lambda \to (\D^0)^*$ is injective, we have $x \in E^0$.
  Then we have from $[x, \mathfrak{h}] = 0$ that $x \in H$. Hence, $H$ is self-centralizing. 
  The rest of (EA2) is trivial.
  Now we describe $R = \mathrm{supp}_{H^*}(E)$ and $R^0 = \{ \alpha \in R \mid (\alpha | \alpha)=0 \}$.
  We consider $\mathfrak{h}^*$ as a subspace of $H^*$ by setting $\langle \alpha, \C^0 \oplus \D^0 \rangle = 0$ for $\alpha \in \mathfrak{h}$,
  and similarly we consider $(\C^0)^*$ and $(\D^0)^*$ as subspaces of $H^*$.
  Then we see that $R \subseteq \mathfrak{h}^* \oplus (\D^0)^*$ since $\C^0$ is central in $E$.
  Moreover, if we view $\Lambda$ as a subgroup of $(\D^0)^*$ using the evaluation map, we have
  \begin{equation} 
    R \subseteq (\Delta \cup \{ 0 \}) \oplus \Lambda \subseteq \mathfrak{h}^* \oplus (\D^0)^*.
  \end{equation}
  For $\alpha + \lambda, \beta + \mu \in R$ where $\alpha, \beta \in \Delta \cup \{ 0 \}$, $\lambda, \mu \in \Lambda$, it is easily checked that
  \[ (\alpha + \lambda | \beta + \mu) = (\alpha | \beta).
  \]
  Hence, we have that $R^0 = R \cap \Lambda$ and $R \setminus R^0 = \{ \alpha +\lambda \in R \mid \alpha \in \Delta, \lambda \in \Lambda \}$.
  Now (EA3) and (EA4) obviously follow since $\Delta$ is an irreducible finite root system.
  Since $\mathfrak{h}=\L^0_0 \neq 0$ and $C(\L) \cong k(\Gamma)$, we have $\Gamma \subseteq \mathrm{supp}_{\Lambda}(\L_0)$.
  Thus, $\Gamma \subseteq \langle R^0 \rangle \subseteq \Lambda$.
  Since $\Gamma$ and $\Lambda$ are both the free abelian groups of rank $n$,
  (EA6) follows.
  Finally, we show (EA5).
  We show first that $\C \subseteq E_c$.
  Take arbitrary $\mu \in \Gamma$.
  For $\lambda \in \Lambda$, we define $c_{\lambda}^{(\mu)} \in \C^{\mu}$ as 
  \begin{equation}
    \begin{cases}
      c_{\lambda}^{(\mu)} (t^{-\mu} \cdot \partial_{\theta}) = \lambda (\theta), \\
      c_{\lambda}^{(\mu)} (\D^{\nu}) = 0 \ \ \mathrm{if} \ \nu \neq -\mu.
    \end{cases} 
  \end{equation} 
  For arbitrary $\alpha \in \Delta$, we take $x \in \L_{\alpha}$, $y \in \L_{-\alpha}$ such that $(x|y)=1$.
  Then we can easily checked for $\lambda \in \Gamma$ that 
  \[ [ t^{\lambda} \cdot x, t^{-\lambda + \mu } \cdot y] - [x, t^{\mu} \cdot y] = c_{\lambda}^{(\mu)} \in E_c.
  \]
  Since $\Gamma$ is rank $n$, $\mathrm{span}_k \{ c_{\lambda}^{(\mu)} \mid \lambda \in \Gamma \} = \C^{\mu}$,
  and hence we have $\C^{\mu} \subseteq E_c$.
  Since $\mu$ is arbitrary, $\C \subseteq E_c$.
  Next, we show that $E_c = \L \oplus \C$.
  $E_c \subseteq \L \oplus \C$ is clear from the definition of the product of $E$.
  Since
  \[ [x, y] \equiv [x, y]_{\L} \ \ \mathrm{mod} \ \C
  \]
  for $x, y \in \L$, we see that $E_c / \C$ is isomorphic to some $\Lambda$-graded ideal in $\L$.
  Since $\L$ is graded-simple, we have $E_c/\C \cong \L$, that is $E_c = \L \oplus \C$.
  Now we show (EA5). 
  Suppose $e \in E$ satisfies $[e, E_c]=0$, and we write $e= l + c + d$ where $l \in \L, c \in \C, d \in \D$.
  We can assume $e \in E^{\lambda}$ for some $\lambda \in \Lambda$.
  Since $[e , \mathfrak{h}]=0$, $e \in E_0^{\lambda}$.
  We can write $d = t^{\lambda} \cdot \partial_{\theta}$ for some $\theta \in \mathrm{Hom}_k (\Lambda, k)$.
  Take $0 \neq y \in \L_{\beta}^{\mu}$ for some $\beta \in \Delta, \mu \in \Lambda$.
  We have 
  \[ 0 = [e, y] = [l, y] +\theta (\mu) t^{\lambda} \cdot y,
  \]
  that is, $[l,y] = - \theta (\mu) t^{\lambda} \cdot y$.
  For $\nu \in \Gamma$, we have
  \[ 0 = [e, t^{\nu} \cdot y] = t^{\nu} \cdot [l, y] + (\theta (\mu)+ \theta (\nu)) t^{\lambda+ \mu} \cdot y 
       = \theta(\nu) t^{\lambda + \nu} \cdot y.
  \]
  Since $\Gamma$ is rank $n$, we see that $\theta = 0$, that is $d = 0$.
  Thus, $e \in \L \oplus \C = E_c$.
\qed
\\

For $(\L, \hz, ( \ | \ ))$ satisfying the conditions (L1)-(L4), we set
\[ \mathcal{P}(\L)=\{ (\D, \tau) \mid \D, \tau \ \mathrm{are \ as \ in \ (i), \ (ii) \ in \ Construction \
   \ref{con}} \}.
\]
Note that $\mathcal{P}(\L)$ does not depend on $\hz$ or $( \ | \ )$. 

We use the following notation: suppose that $(\L, \hz, ( \ | \ ))$ and $(\L', \hz', ( \ | \ )')$ 
satisfy the conditions (L1)-(L4).
Then we will write 
\[ (\L, \hz, ( \ | \ )) \sim_{\mathrm{EALA}} (\L', \hz', ( \ | \ )') \ (\mathrm{or} \ \L \sim_{\mathrm{EALA}} \L'
   \ \mathrm{for \ short})
\]
if there exists a bijection $\mathcal{P}(\L) \to \mathcal{P}(\L')$
such that $E(\L,\D,\tau)$ is isomorphic as EALAs to $E(\L',\D', \tau')$
where $(\D', \tau') \in \mathcal{P}(\L')$ 
is the image of $(\D, \tau) \in \mathcal{P}(\L)$ under the bijection.
In other words, $\L \sim_{\mathrm{EALA}} \L'$ means that $\{ E(\L, \D, \tau) | (\D, \tau) \in \mathcal{P}(\L) \}$ and 
$\{ E( \L', \D', \tau') | (\D', \tau') \in \mathcal{P}(\L') \}$ coincide up to isomorphism of EALAs. 

Using the above notation, we have the following:
\begin{lem} \label{lem of EALA}
  Suppose that $(\L, \hz, ( \ | \ ))$ satisfies the conditions \textup{(L1)}-\textup{(L4)} and
  we set $Q_{\hz}=\sum_{\alpha \in \Delta} \mathbb{Z}\alpha$.
  \textup{(a)} Let $s \in \mathrm{Hom}(Q_{\hz}, \Lambda)$.
               For a suitable bilinear form $( \ | \ )^{(s)}$,
               $\left(\L^{(s)}, \hz, ( \ | \ )^{(s)}\right)$ also satisfies the conditions \textup{(L1)}-\textup{(L4)},
               and $\L \sim_{\mathrm{EALA}} \L^{(s)}$. 
  \textup{(b)} Let $\rho: \lr{\Lambda}{\L} \to \Lambda$ be an injective homomorphism.
               For a suitable bilinear form $( \ | \ )_{(\rho)}$ on $\L_{(\rho)},$ 
               $\left(\L_{(\rho)}, \hz, ( \ | \ )_{(\rho)}\right)$ also 
               satisfies the conditions \textup{(L1)}-\textup{(L4)},
               and $\L \sim_{\mathrm{EALA}} \L_{(\rho)}$.  
\end{lem}
{\textbf{Proof.}}
  (a) Since $\L= \L^{(s)}$ as a Lie algebra, we can view $( \ | \ )$ as a bilinear form on $\L^{(s)}$.
      Let $( \ | \ )^{(s)}$ be this bilinear form.
      Then it is easily checked that $\left(\L^{(s)}, \hz, ( \ | \ )^{(s)}\right)$ satisfies (L2)-(L4).
      To show that $\L^{(s)}$ is graded-central-simple $\Lambda$-graded,
      suppose that $I \subseteq \L^{(s)}$ is a $\Lambda$-graded ideal.
      Then $I$ is $Q_{\hz} \times \Lambda$-graded by (L4).
      By considering $I$ as a ideal of $\L$, we can see that $I = \{ 0 \}$ or $\L^{(s)}$.
      Also $C(\L^{(s)})^0= k \cdot \mathrm{id}$ is clear, and hence $\L^{(s)}$ satisfies (L1).
      The second statement can be proved in exactly the same way as \cite[Corollary 6.3]{MR2743759}. \\
  (b) Since $\L=\L_{(\rho)}$ as a Lie algebra, we can view the identity on $\L$ as an isomorphism from $\L$
      onto $\L_{(\rho)}$.
      We denote this isomorphism by $\psi: \L \to \L_{(\rho)}$, 
      and define a bilinear form $( \ | \ )_{(\rho)}$
      on $\L_{(\rho)}$ as $(\psi(x)|\psi(y))_{(\rho)}=(x|y)$ for $x,y \in \L$.
      Note that $\Gamma \subseteq \langle \mathrm{supp}_\Lambda (\L) \rangle$ as stated in the proof of Proposition \ref{construction of an EALA}.
      From the definition of $\L_{(\rho)}$, the central grading group of $\L_{(\rho)}$ is $\rho(\Gamma)$ .
      Thus it is easily checked that $\left( \L_{(\rho)}, \hz, ( \ | \ )_{(\rho)}\right)$ satisfies (L1)-(L4).
      To show that $\L \sim_{\mathrm{EALA}} \L_{(\rho)}$, we first define a map
      \[ \mathcal{P}(\L) \ni (\D, \tau) \mapsto (\D_{(\rho)}, \tau_{(\rho)}) \in \mathcal{P}(\L_{(\rho)}).
      \]
      As in the Construction \ref{con}, 
      we write $C(\L)=\oplus_{\mu \in \Gamma} \, k t^{\mu}$.
      Then we can write $C(\L_{(\rho)})=\oplus_{\mu \in \Gamma} \, k s^{\rho(\mu)}$ where 
      \[ s^{\rho(\mu)} \cdot \psi(x) = \psi(t^{\mu} \cdot x)
      \]
      for $\mu \in \Gamma, x \in \L$.
      Let $\theta \in \mathrm{Hom}(\Lambda, k)$.
      Since $\mathrm{Im} \, \rho=\lr{\Lambda}{\L_{(\rho)}}$, we can define $\theta \circ \rho^{-1}$ 
      as a homomorphism from $\lr{\Lambda}{\L_{(\rho)}}$ to $k$.
      Since the rank of $\lr{\Lambda}{\L_{(\rho)}}$ is $n$, 
      there exists unique homomorphism $\widetilde{\theta \circ \rho^{-1}} \in \mathrm{Hom}(\Lambda, k)$
      such that $\widetilde{\theta \circ \rho^{-1}} \big|_{\lr{\Lambda}{\L_{(\rho)}}}=\theta \circ \rho^{-1}$.
      Using this notation, we define a $k$-linear isomorphism 
      $\omega: \mathrm{SCDer}(\L) \to \mathrm{SCDer}(\L_{(\rho)})$ as
      \[ \omega(t^{\mu} \partial_{\theta})= s^{\rho(\mu)} \partial_{\widetilde{\theta \circ \rho^{-1}}}.
      \]
      Since
      \[ \begin{split}
           [s^{\rho(\mu_1)} \partial_{\widetilde{\theta_1 \circ \rho^{-1}}}, 
           s^{\rho(\mu_2)} \partial_{\widetilde{\theta_2 \circ \rho^{-1}}}] &=
           \theta_1(\mu_2)s^{\rho(\mu_1+ \mu_2)} \partial_{\widetilde{\theta_2 \circ \rho^{-1}}}
           - \theta_2(\mu_1) s^{\rho(\mu_1+\mu_2)} \partial_{\widetilde{\theta_1 \circ \rho^{-1}}} \\
           & = \theta_1(\mu_2) \omega (t^{\mu_1+ \mu_2} \partial_{\theta_2})- 
               \theta_2(\mu_1) \omega(t^{\mu_1 + \mu_2} \partial_{\theta_1}),
         \end{split}
      \]
      $\omega$ is a Lie algebra isomorphism.
      We have 
      \begin{equation} \label{74} 
        \omega(d)(\psi(x)) = \psi(d(x))
      \end{equation}
      for $d \in \D, x \in \L$ since if $y \in \L^{\lambda}$ for $\lambda \in \mathrm{supp}_{\Lambda}(\L)$,
      \[ \begin{split}
          \omega(t^{\mu} \partial_{\theta})(\psi(y)) &= 
          s^{\rho(\mu)} \partial_{\widetilde{\theta \circ \rho^{-1}}} \big(\psi(y) \big) = 
           \widetilde{\theta \circ \rho^{-1}}\big(\rho(\lambda)\big)s^{\rho(\mu)} \cdot \psi(y) \\
           & = \theta(\lambda)\psi (t^{\mu} \cdot y) = \psi\big(t^{\mu} \partial_{\theta}(y)\big).
         \end{split}
      \]
      We put $\D_{(\rho)}=\omega(\D)$, and set $\C_{(\rho)}=\oplus_{\mu \in \Gamma} \big(\D_{(\rho)}^{\rho(\mu)}\big)^*$.
      We define $\hat{\omega}: \C \to \C_{(\rho)}$ by 
      \[ \hat{\omega}(c)(\omega(d))=c(d)
      \]
      for $c \in \C, d \in \D$,
      and define $\tau_{(\rho)}: \D_{(\rho)} \times \D_{(\rho)} \to \C_{(\rho)}$ as
      \[ \tau_{(\rho)}\big(\omega(d_1), \omega(d_2)\big)\big(\omega(d_3)\big)=\tau(d_1, d_2)(d_3)
      \]
      for $d_i \in \D$. Then $(\D_{(\rho)}, \tau_{(\rho)}) \in \mathcal{P}(\L_{(\rho)})$ is clear.
      Next, we show that the map $x + c + d \mapsto \psi(x) + \hat{\omega}(c) + \omega(d)$ 
      for $x \in \L, c \in \C, d \in \D$ is a Lie algebra isomorphism.
      To prove this fact, it suffices to show that
      \begin{equation} \label{75}
        \sigma_{\D_{(\rho)}}\big(\psi(x_1), \psi(x_2)\big)= \hat{\omega}\big(\sigma_{\D}(x_1, x_2)\big)
      \end{equation}
      for $x_i \in \L$ since we have using (\ref{74}) that 
      \[ \begin{split}
           & [\psi(x_1) + \hat{\omega}(c_1) + \omega(d_1), \psi(x_2) + \hat{\omega}(c_2)+\omega(d_2)] \\
           & =\Big([\psi(x_1), \psi(x_2)] + \omega(d_1)(\psi(x_2)) - \omega(d_2)(\psi(x_1))\Big) \\
           & + \Big(\sigma_{\D_{(\rho)}}\big(\psi(x_1), \psi(x_2)\big) + \omega(d_1) \cdot \hat{\omega}(c_2)
             - \omega(d_2) \cdot \hat{\omega}(c_1) +\tau_{(\rho)}\big(\omega(d_1), \omega(d_2)\big)\Big) \\
           & + [\omega(d_1), \omega(d_2)] \\
           & = \Big(\psi([x_1, x_2]) + \psi(d_1(x_2)) - \psi(d_2(x_1))\Big) \\
           & + \Big(\sigma_{\D_{(\rho)}}\big(\psi(x_1), \psi(x_2)\big) + \hat{\omega}(d_1 \cdot c_2)
             - \hat{\omega}(d_2 \cdot c_1) + \hat{\omega}\big(\tau(d_1, d_2)\big)\Big) + \omega([d_1, d_2]).
         \end{split}
      \]
      (\ref{75}) follows since 
      \[ \begin{split}
           \sigma_{\D_{(\rho)}}\big(\psi(x_1), \psi(x_2)\big)\big(\omega(d)\big) &= \big(\omega(d)(\psi(x_1)) \big|
           \psi(x_2)\big)_{(\rho)}
           = ( \psi(d(x_1))| \psi(x_2))_{(\rho)} \\ &= (d(x_1)| x_2) = \sigma_{\D}(x_1, x_2)(d).
         \end{split}
      \]
      It is easy to see that this isomorphism preserves the bilinear forms and sends $H$ to
      $H_{(\rho)}=\hz \oplus \C_{(\rho)} \oplus \D_{(\rho)}$.
      Finally, to show that the map $(\D, \tau) \mapsto (\D_{(\rho)}, \tau_{(\rho)})$ is bijective,
      we construct the inverse of this map.
      By (\ref{9'}), $\rho$ induces a group isomorphism 
      $\bar{\rho}: \lr{\Lambda}{\L} \to \lr{\Lambda}{\L_{(\rho)}}$.
      For the canonical injective homomorphism $\iota: \lr{\Lambda}{\L} \to \Lambda$, 
      it is easily checked that
      \[ ( \L_{(\rho)})_{(\iota \circ \bar{\rho}^{-1})} = \L.
      \]
      Then we can see that 
      \[ \mathcal{P} ( \L_{(\rho)}) \ni (\D', \tau') \mapsto (\D'_{(\iota \circ \bar{\rho}^{-1})}, 
         \tau'_{(\iota \circ \bar{\rho}^{-1})}) \in \mathcal{P}(\L)
      \]
      is the inverse of the map $(\D, \tau) \mapsto (\D_{(\rho)}, \tau_{(\rho)})$.
\qed      
\\

Let $\L= L_{\bm{m}}(\gn, \bm{\sigma}, \hz)$ be a multiloop Lie algebra of nullity $n$ such that $\gsigma \neq
\{ 0 \}$.
Let $( \ | \ )$ be the Killing form of $\gn$, and define a non-degenerate, invariant, symmetric, 
$\mathbb{Z}^n$-graded bilinear form on $\L$ (which we also write as $( \ | \ )$) by
\begin{equation} \label{76}
  (x \otimes t^{\lambda} | y \otimes t^{\mu}) = 
    \begin{cases}
      (x|y) & \text{if $\lambda + \mu = 0$}, \\
      0     & \text{otherwise}
    \end{cases}
\end{equation}
where $\lambda, \mu \in \mathbb{Z}^n, x \in \gn^{\bar{\lambda}}, y \in \gn^{\bar{\mu}}$.
Then, $(\L, \hz, ( \ | \ ))$ satisfies (L1)-(L4) by Lemma \ref{central grading group} and Corollary \ref{cor1}.
The following proposition shows that a bilinear form on $\L$  satisfying (L3) is only that defined in (\ref{76})
up to a scalar multiplication.
\begin{prop} \label{prop of bilinear form}
  Suppose that a bilinear form $( \ | \ )'$ on $\L$ is non-degenerate, invariant, symmetric, 
  and $\mathbb{Z}^n$-graded.
  Then we have $( \ | \ )'= c ( \ | \ )$ for $0 \neq c \in k$, where $( \ | \ )$ is the bilinear form defined in 
  \textup{(\ref{76})}.
\end{prop}
{\textbf{Proof.}} 
  We write $xt^{\lambda}=x \otimes t^{\lambda} \in \L$.
  For each $\alpha \in \Delta$, we take an $\mathfrak{sl}_2(k)$-triple $\{ x_{\alpha}^{\bar{\lambda}_{\alpha}},
  x_{-\alpha}^{-\bar{\lambda}_{\alpha}}, h_{\alpha} \}$ for some $\lambda_{\alpha} \in \mathbb{Z}^n$.
  We choose $\gamma \in \Delta$ arbitrarily, 
  and suppose that $( h_{\gamma}|h_{\gamma})'=c(h_{\gamma}|h_{\gamma})$.
  If $\beta \in \Delta$ satisfies $( h_{\beta}|h_{\gamma}) \neq 0$,
  \[ \begin{split}
       (h_{\beta}|h_{\beta})'& = (h_{\beta}|[x_{\beta}^{\bar{\lambda}_{\beta}}t^{\lambda_{\beta}},
       x_{-\beta}^{-\bar{\lambda}_{\beta}}t^{-{\lambda}_{\beta}}])'=2(x_{\beta}^{\bar{\lambda}_{\beta}}
       t^{\lambda_{\beta}}| x_{-\beta}^{-\bar{\lambda}_{\beta}}t^{-\lambda_{\beta}})' \\
       & = \frac{2}{\langle \beta, h_{\gamma} \rangle} ([h_{\gamma}, x_{\beta}^{\bar{\lambda}_{\beta}}
       t^{\lambda_{\beta}}]|x_{-\beta}^{-\bar{\lambda}_{\beta}}t^{-{\lambda}_{\beta}})' 
       = \frac{(h_{\beta}|h_{\beta})}{(h_{\beta}|h_{\gamma})}(h_{\gamma}|h_{\beta})' \\
       & =\frac{(h_{\beta}|h_{\beta})}{(h_{\beta}|h_{\gamma})}([x_{\gamma}^{\bar{\lambda}_{\gamma}}
       t^{\lambda_{\gamma}}, x_{-\gamma}^{-\bar{\lambda}_{\gamma}}t^{-\lambda_{\gamma}}]|h_{\beta})' 
       = \frac{2(h_{\beta}|h_{\beta})}{(h_{\gamma}|h_{\gamma})}(x_{\gamma}^{\bar{\lambda}_{\gamma}}
       t^{\lambda_{\gamma}}|x_{-\gamma}^{-\bar{\lambda}_{\gamma}}t^{-\lambda_{\gamma}})' \\
       &=\frac{(h_{\beta}|h_{\beta})}{(h_{\gamma}|h_{\gamma})}(h_{\gamma}|h_{\gamma})'
       =c(h_{\beta}|h_{\beta}).
     \end{split}
  \]
  By repeating calculations as above, we have $(h_{\alpha}|h_{\alpha})'=c(h_{\alpha}|h_{\alpha})$
  for any $\alpha \in \Delta$ since $\Delta$ is irreducible. 
  Then for arbitrary $\alpha \in \Delta, \lambda \in \mathbb{Z}^n$ and $x \in \gn_{\alpha}^{\bar{\lambda}},
  y \in \gn_{-\alpha}^{-\bar{\lambda}}$,
  \[ \begin{split} (xt^{\lambda}|yt^{-\lambda})' & =\frac{1}{2}([h_{\alpha}, xt^{\lambda}]|yt^{-\lambda})'
     =\frac{1}{2}(h_{\alpha}|[xt^{\lambda}, yt^{-\lambda}])' \\ 
     & =\frac{(x|y)}{2}(h_{\alpha}|\frac{(\alpha|\alpha)}{2}h_{\alpha})'
     \quad \quad \mathrm{(by \ (\ref{50}) \ and \ (\ref{24}))} \\
     & = c(x|y) = c(xt^{\lambda}|yt^{-\lambda}).
     \end{split}
  \]
  From this, we have $( \ | \ )' = c ( \ | \ )$ on $\bigoplus_{\alpha \in \Delta} \L_{\alpha}$. 
  Then we have $( \ | \ )' = c ( \ | \ )$ on $\L$ since $\L_0 \subseteq \bigoplus_{\alpha \in \Delta} 
  [\L_{\alpha}, \L_{-\alpha}]$ and both $( \ | \ )$ and $( \ | \ )'$ are invariant.
\qed  
\begin{rem} \label{independence of bilinear form} \normalfont
  Suppose that $\L$ is a multiloop Lie algebra and $(\D, \tau) \in \mathcal{P}(\L)$.
  By Proposition \ref{prop of bilinear form}, it is easily checked that $E(\L, \D, \tau)$
  does not depend on the bilinear form used in the construction up to isomorphism as EALAs.
\end{rem}
Now, we can easily show the following theorem:
\begin{thm} \label{thm of EALA construction1}
  Let $\L= \Lmh$ and $\L'=\Lmhb$ be multiloop Lie algebras of nullity $n$,
  and suppose that $\gsigma \neq \{ 0 \}$ and $\gn'^{\bm{\sigma}'} \neq \{ 0 \}$.
  If $\L \suppisom \L'$, then there exists a bijection $\mathcal{P}(\L) \to \mathcal{P}(\L')$
  such that $E(\L,\D,\tau)$ is isomorphic as EALAs to $E(\L',\D', \tau')$
  where $(\D', \tau') \in \mathcal{P}(\L')$ 
  is the image of $(\D, \tau) \in \mathcal{P}(\L)$ under this bijection.
\end{thm}
{\textbf{Proof.}}
  By Theorem \ref{thm for supp-isom} and Lemma \ref{lem of EALA}, there exists a $Q_{\hz} \times \mathbb{Z}^n$-graded
  Lie algebra $\L_{p}$ such that $\L_p \cong_{\mathbb{Z}^n- \mathrm{ig}} \L'$ and 
  $\L \sim_{\mathrm{EALA}} \L_p$.
  Using Lemma \ref{prop3}, it is easily checked that $\L_p \sim_{\mathrm{EALA}} \L'$. 
  Thus we have $\L \sim_{\mathrm{EALA}} \L'$.
\qed 
\\

We prove the following lemma using \cite[Theorem 6.1]{MR2743759}:
\begin{lem} \label{supp-isom of multiloop Lie tori}
  Let $\Lb$ and $\Lb'$ be multiloop Lie $\mathbb{Z}^n$-tori.
  If $E(\Lb, \D, \tau)$ is isomorphic as EALAs to $E(\Lb', \D', \tau')$ 
  for some $(\D, \tau) \in \mathcal{P}(\Lb)$ and $(\D', \tau') \in \mathcal{P}(\Lb')$, 
  then $\mathcal{L} \suppisom \mathcal{L}'$.
\end{lem}
{\textbf{Proof.}}
  Let $Q$ (resp. $Q'$) be the root lattice (i.e. $\mathbb{Z}$-span of its root system) of $\Lb$ (resp. $\Lb'$).
  By \cite[Theorem 6.1]{MR2743759}, there exists $s \in \mathrm{Hom}(Q, \mathbb{Z}^n)$, a Lie algebra isomorphism
  $\varphi: \Lb^{(s)} \to \Lb'$ and two group isomorphisms $\varphi_Q: Q \to Q'$, $\varphi_{\mathbb{Z}^n}:
  \mathbb{Z}^n \to \mathbb{Z}^n$ such that 
  \[ \varphi \big( (\Lb^{(s)})_{\alpha}^{\lambda} \big) = {\Lb'}_{\varphi_Q(\alpha)}^{\varphi_{\mathbb{Z}^n}(\lambda)}
  \]
  for $\alpha \in Q, \lambda \in \mathbb{Z}^n$ (in \cite{MR2743759}, this equivalence relation is called isotopy).
  Then $\Lb \cong_{\mathrm{supp}} \Lb'$ follows. (See the proof of Theorem \ref{thm for supp-isom} (c) $\Rightarrow$ (a).)
\qed \\

Using Theorem \ref{Main theorem}, we can extend this lemma to multiloop Lie algebras 
in which the $0$-homogeneous spaces are non-zero.
\begin{thm} \label{thm of EALA construction2}
  Let $\L=\Lmh$ and $\L'=\Lmhb$ be multiloop Lie algebras of nullity $n$, and
  suppose that $\gsigma \neq \{ 0 \}$ and ${\gn'}^{\bm{\sigma}'} \neq \{ 0 \}$.
  If $E(\L, \D, \tau)$ is isomorphic as EALAs to $E(\L', \D', \tau')$ 
  for some $(\D, \tau) \in \mathcal{P}(\L)$ and $(\D', \tau') \in \mathcal{P}(\L')$, 
  then $\L \suppisom \L'$.
\end{thm}
{\textbf{Proof.}}
  By Theorem \ref{Main theorem}, there exists a multiloop Lie $\mathbb{Z}^n$-torus $\Lb$ (resp.\ $\Lb'$) such that
  $\L \suppisom \Lb$ (resp.\ $\L' \suppisom \Lb'$).
  Then by Theorem \ref{thm of EALA construction1}, there exists $(\bar{\D}, \bar{\tau}) \in \mathcal{P}(\Lb)$
  (resp.\ $(\bar{\D}', \bar{\tau}') \in \mathcal{P}(\Lb')$) such that 
  $E(\L, \D, \tau)$ and $E(\Lb, \bar{\D}, \bar{\tau})$ (resp.\ $E(\L', \D', \tau')$ and $E(\Lb', \bar{\D}', \bar{\tau}')$) 
  are isomorphic as EALAs.
  Therefore, $E(\Lb, \bar{\D}, \bar{\tau})$ and $E(\Lb', \bar{\D}', \bar{\tau}')$ are isomorphic
  as EALAs, and then $\Lb \suppisom \Lb'$ by Lemma \ref{supp-isom of multiloop Lie tori}.
  Thus, we have $\L \suppisom \L'$.
\qed \\

\textbf{Acknowledgements:} The author is very grateful to his supervisor Hisayosi Matumoto
for his constant encouragement and patient guidance. 
He is supported by the Japan Society for the Promotion of Science Research
Fellowships for Young Scientists.

\def\cprime{$'$} \def\cprime{$'$}


\end{document}